\documentclass[12pt]{article}
 
\usepackage{myarticle}
\usepackage[colorinlistoftodos,prependcaption,textsize=tiny,backgroundcolor=black!5!white,bordercolor=red]{todonotes}

\usepackage{pgfplots}
\pgfplotsset{
  tick label style={font=\footnotesize},
  label style={font=\footnotesize},
  legend style={font=\footnotesize}
}

\theoremstyle{plain}
\newtheorem{theorem}{Theorem}
\newtheorem{lemma}[theorem]{Lemma}
\newtheorem{proposition}[theorem]{Proposition}

\newtheorem{example}[theorem]{Example}

\declaretheoremstyle[
bodyfont=\normalfont\itshape,
headformat=\NAME\NUMBER
]{nospacetheorem}

\usepackage[colorinlistoftodos,prependcaption,textsize=tiny]{todonotes}

\newcommand{\x}{\bm{x}}
\renewcommand{\b}{\bm{b}}
\renewcommand{\c}{\bm{c}}

\newcommand{\y}{\bm{y}}
\newcommand{\z}{\bm{z}}

\newcommand{\xz}{\bm{x}^{(0)}}
\newcommand{\xk}{\bm{x}^{(k)}}
\newcommand{\xn}{\bm{x}^{(n)}}
\newcommand{\xtk}{x^{(k)}}
\newcommand{\xK}{\bm{x}^{(K)}}
\newcommand{\xkp}{\bm{x}^{(k+1)}}

\newcommand{\yK}{\bm{y}^{(K)}}
\newcommand{\Dut}{D_{u}}

\newcommand{\gn}{\bm{g}^{(n)}}
\newcommand{\gtk}{g^{(k)}}
\newcommand{\gK}{\bm{g}^{(K)}}
\newcommand{\xtmin}{x^{*}}
\newcommand{\xmin}{\bm{x}^{*}}

\renewcommand{\v}{\bm{v}}
\renewcommand{\u}{\bm{u}}
\renewcommand{\b}{\bm{b}}
\renewcommand{\c}{\bm{c}}

\renewcommand{\d}{\bm{d}}
\newcommand{\N}{\mathbb{N}}

\newcommand{\BigOh}{\mathcal{O}}
\renewcommand{\P}{\mathcal{P}}
\newcommand{\D}{\mathcal{D}}
\newcommand{\U}{\mathcal{U}}

\newcommand{\X}{\mathcal{X}}

\newcommand{\RW}{\mathcal{R}}
\newcommand{\R}{\mathbb{R}}
\newcommand{\Rext}{\overline{\mathbb{R}}}
\newcommand{\RN}{\mathbb{R}^{N}}
\newcommand{\RP}{\mathbb{R}^{P}}
\newcommand{\RM}{\mathbb{R}^{M}}
\newcommand{\RNXRP}{\mathbb{R}^{N}\times\mathbb{R}^{P}}

\newcommand{\RNXP}{\mathbb{R}^{N \times P}}

\newcommand{\Aadj}{A^{*}}
\newcommand{\Badj}{B^{*}}
\newcommand{\pc}{p^{*}}
\newcommand{\pcc}{p^{**}}
\newcommand{\fc}{f^{*}}
\newcommand{\hc}{h^{*}}
\newcommand{\kc}{k^{*}}
\newcommand{\gc}{g^{*}}
\newcommand{\abs}[1]{\left\lvert #1 \right\rvert}
\newcommand{\norm}[1]{\left\lVert #1 \right\rVert}

\newcommand{\innerprod}[2]{\langle #1 , #2 \rangle}
\newcommand{\sgn}[1]{\mathrm{sgn}(#1)}

\newcommand{\Du}{D_{\bm{u}}}

\newcommand{\grad}[2]{\nabla_{#1} #2}
\newcommand{\subdiff}[2]{\partial_{#1} #2}
\newcommand{\hess}[2]{\nabla_{#1}^{2} #2}

\newcommand{\dom}[1]{\textnormal{dom}\:#1}

\newcommand{\ri}[1]{\textnormal{ri}\:#1}
\newcommand{\intr}[1]{\textnormal{int}\:#1}

\newcommand{\seq}[2]{(#1^{(#2)})_{#2 \in \N}}

\newcommand{\fref}[1]{Figure~\ref{#1}}
\newcommand{\sref}[1]{Section~\ref{#1}}
\newcommand{\exref}[1]{Example~\ref{#1}}
\newcommand{\tref}[1]{Theorem~\ref{#1}}
\newcommand{\lref}[1]{Lemma~\ref{#1}}
\DeclareMathOperator*{\argmax}{arg\,max}
\DeclareMathOperator*{\argmin}{arg\,min}
\newcommand{\st}{\mathrm{s.t.}}
\newcommand{\figlen}{0.98}
\newcommand{\figcelllen}{0.245}

\graphicspath{{./figures/}}

\begin{document}

\thispagestyle{empty}
\begin{center}
\vspace*{0.03\paperheight}
{\Large\bf Differentiating the Value Function by using \\[1ex]
Convex Duality} \\
\bigskip
\bigskip
{\large Sheheryar Mehmood and Peter Ochs \\ \medskip
{\small
University of T\"{u}bingen, T\"{u}bingen, Germany \\
}
}
\end{center}
\bigskip

\begin{abstract}
We consider the differentiation of the value function for parametric optimization problems. Such problems are ubiquitous in Machine Learning applications such as structured support vector machines, matrix factorization and min-min or minimax problems in general. Existing approaches for computing the derivative rely on strong assumptions of the parametric function. Therefore, in several scenarios there is no theoretical evidence that a given algorithmic differentiation strategy computes the true gradient information of the value function. We leverage a well known result from convex duality theory to relax the conditions and to derive convergence rates of the derivative approximation for several classes of parametric optimization problems in Machine Learning. We demonstrate the versatility of our approach in several experiments, including non-smooth parametric functions. Even in settings where other approaches are applicable, our duality based strategy shows a favorable performance.
\end{abstract}




\section{Introduction}

\begin{figure}[t]
\begin{subfigure}{.32\textwidth}
    \centering
    \includegraphics[width=\linewidth]{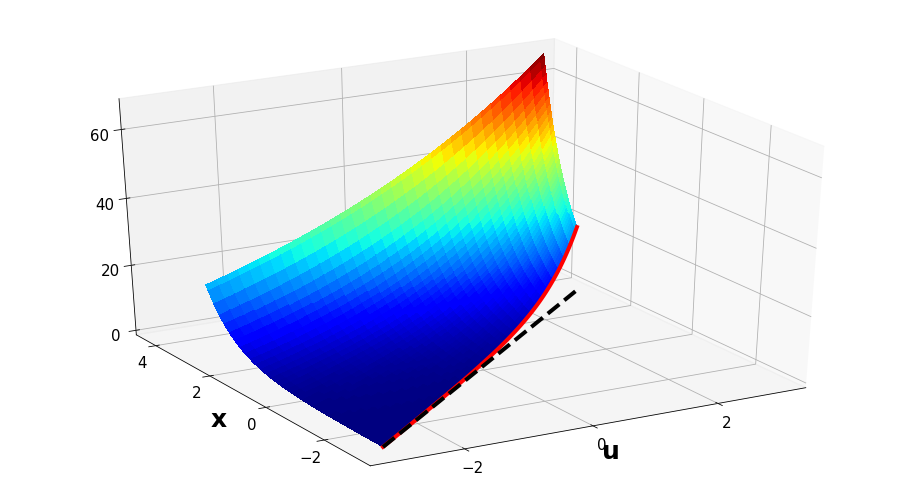}
\end{subfigure}
\begin{subfigure}{.32\textwidth}
    \centering
    \includegraphics[width=\linewidth]{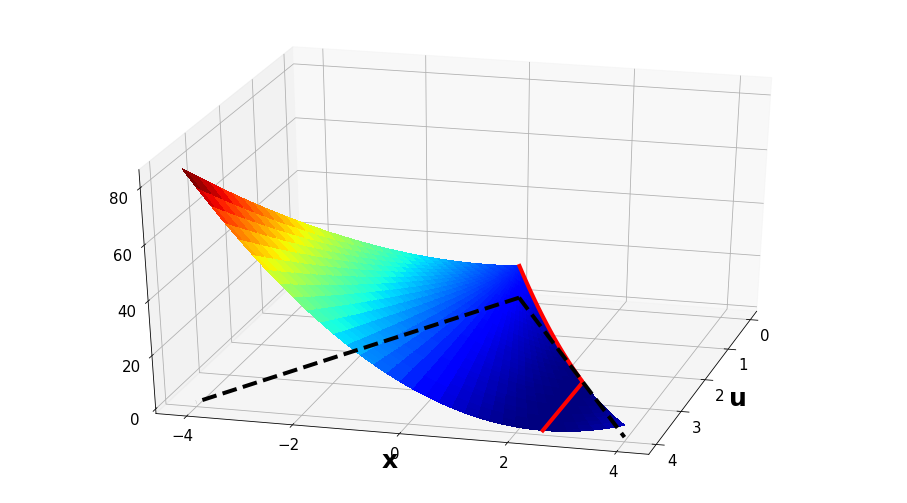}
\end{subfigure}
\begin{subfigure}{.32\textwidth}
    \centering
    \includegraphics[width=\linewidth]{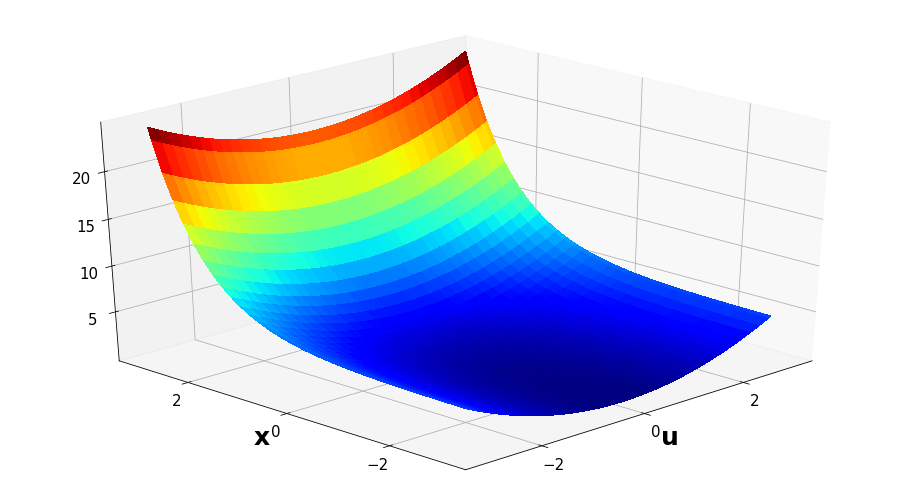}
\end{subfigure}
\begin{subfigure}{.32\textwidth}
    \centering
    \includegraphics[width=\linewidth]{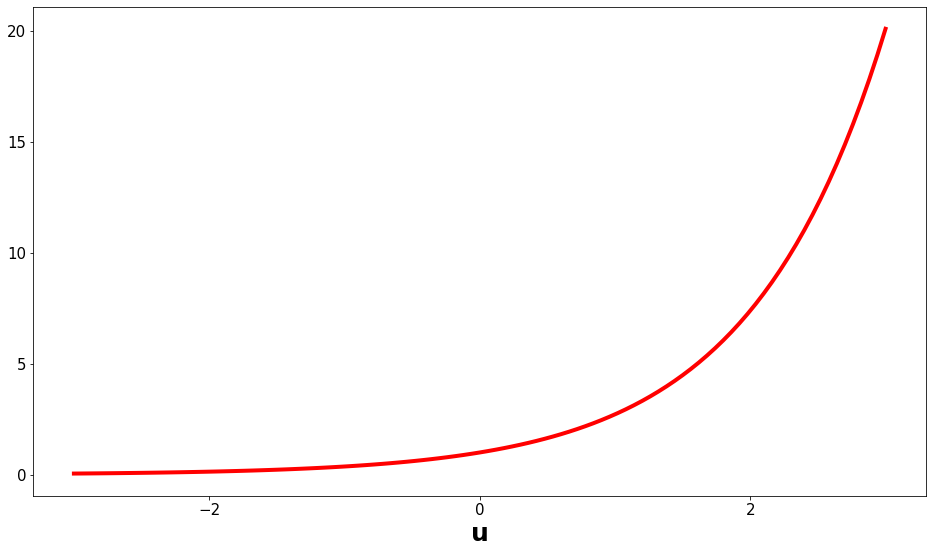}
\end{subfigure}
\begin{subfigure}{.32\textwidth}
    \centering
    \includegraphics[width=\linewidth]{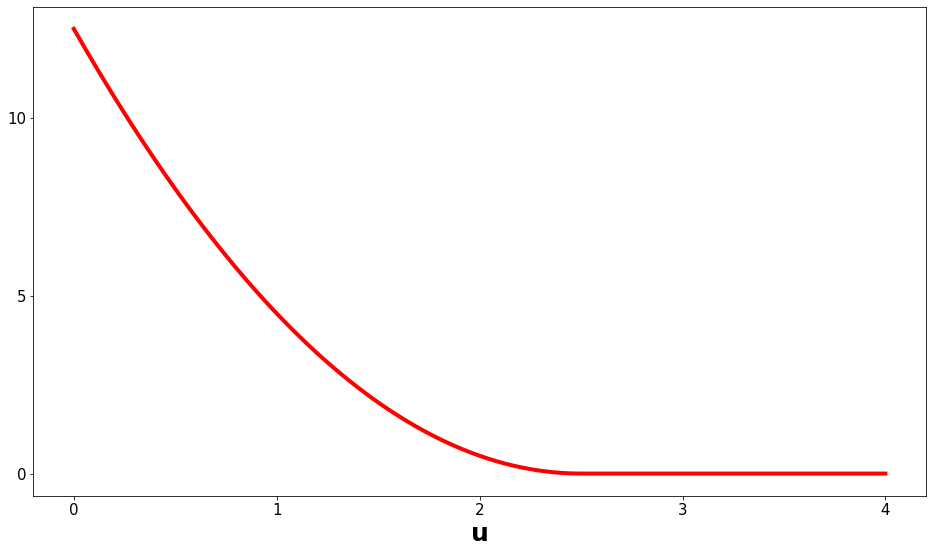}
\end{subfigure}
\begin{subfigure}{.32\textwidth}
    \centering
    \includegraphics[width=\linewidth]{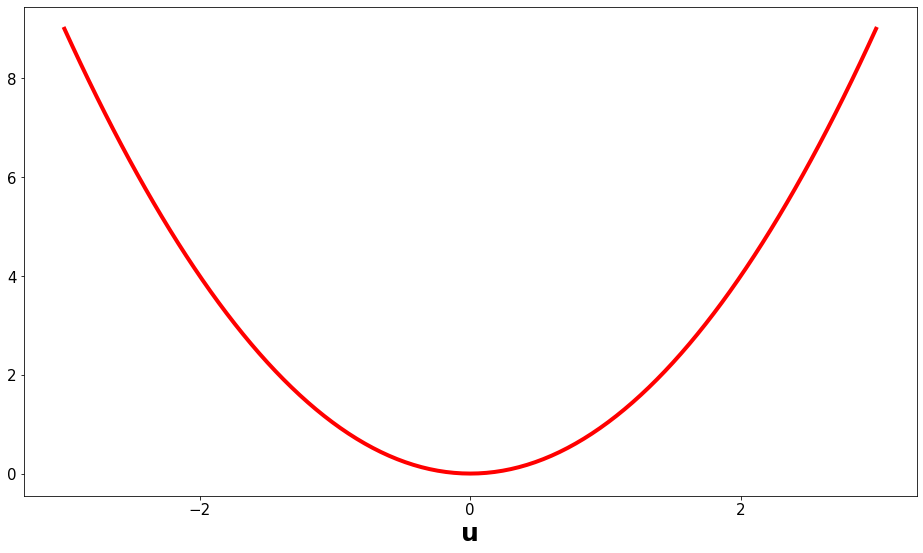}
\end{subfigure}
\caption{Graphical Depiction of some simple objective functions which have a smooth value function but either don't have a minimum or are not differentiable at the minimizer. (Top row) The objective $ f $ is plotted against $ x $ and $ u $ for some toy examples (left) $ \R \times \R \ni (x, u) \mapsto f(x, u) = \exp(x) + \delta_{[u, +\infty)} (x) $, (middle) $ \R \times [0, 4] \ni (x, u) \mapsto f(x, u) = (a x - b)^{2}/2 + \delta_{[-u, u]} (x) $ where $ a, b \neq 0 $ and (right) $ \R \times \R \ni (x, u) \mapsto f(x, u) = \exp(x) + u^{2} / 2 $. The function $u \mapsto (u, \xtmin, f(\xtmin, u)) $ is sketched in red in first two examples (left and middle). No such curve is shown for the third example because the minimum w.r.t. $ x $ is not attained. The black dashed-line (left and middle) shows the boundary of the feasible set in $ \R \times \R $. (Bottom row) Value function $ p $ plotted against $ u $ for the corresponding examples.}
\label{fig:IntroEx}
\end{figure}

Given a function $ f : \RNXRP \to \R $ with values $ f(\x, \u) $, we consider the following parametric optimization problem:
\begin{equation} \tag{$ \P $} \label{eq:PrimalProb}
    p(\u) \coloneqq \inf_{\x \in \RN} f(\x, \u) \,.
\end{equation}
The optimal value of $ f(\cdot, \u) $, which we denote by $ p (\u) $, depends on the parameter $ \u $ and is commonly referred to as the \textit{value function} of \eqref{eq:PrimalProb} or the \textit{infimal projection} of $ f $. When the minimum is attained at some $ \xmin(\u) \in \RN $ for a given $ \u \in \RP $, the value function is given by $ p(\u) = f(\xmin(\u), \u) $. For many applications, quantifying the change in $ p $ with respect to $ \u $ is key, which is achieved by computing gradient or subgradient information of $  p $.

This is particularly true for Machine Learning applications, for which a parametric dependency occurs naturally, for example, when solving a min-min or minimax optimization problem, in Structured Support Vector Machines \cite{TGK04, TJHA05}, Sparse Dictionary Learning \cite{MBPS10}, Generative Adversarial Networks \cite{GPM+14} and Matrix Factorization. Another important area where such derivative information is crucial is the Sensitivity Analysis of an optimization problem, which finds applications in the shadow price problem \cite[Section 4.3]{Sti18} and also in bridge crane design or breakwater modeling \cite{CMC08}. The decision-making is based on a measure of how sensitive the model is when parameters $ \u $ are changed.

If $\xmin(\u)$ is available and differentiable, the gradient information can be computed by differentiating $ p(\u) = f(\xmin(\u), \u) $ with respect to $ \u $, i.e., 
\begin{equation*}
    \grad{}{p}(\u) = \Du \xmin (\u) \grad{\x}{f} (\xmin(\u), \u) + \grad{\u}{f} (\xmin(\u), \u) = \grad{\u}{f} (\xmin(\u), \u) \,.
\end{equation*}
However, clearly, this approach demands for strong smoothness conditions of the parametric function $ f $ and the solution mapping $ \xmin(\u) $, which are not satisfied for common Machine Learning applications.

Consider for example the following sparsity constrained linear regression problem:
\begin{equation} \label{eq:Lasso}
    \min_{\x} \norm{A \x- b}_2\,,\ \st\ \norm{\x}_1 \leq u\,,
\end{equation}
where $A\in \R^{M\times N}$, $b\in \R^M$, and $u\geq 0$. As a constrained optimization problem, the objective (including the constraint in terms of an indicator function) is not differentiable. The value function, however, is continuously differentiable on $ (0, \infty) $ and subdifferentiable at $ u=0 $ \cite{BF08}. As noted in \cite{BF08} and more generally in \cite{BF11, ABF13}, its gradient can be used to solve the following minimial norm problem:
\begin{equation} \label{eq:LassoEqv}
    \min_{\x} \norm{\x}_1 \,,\ \st\ \norm{A \x - b}_2 \leq w\,.
\end{equation}
The problem in \eqref{eq:Lasso} is one of many instances where the parametric function $ f $ is jointly convex in its arguments. Yet algorithmic differentiation strategies based on differentiating approximations to the solution mapping cannot be applied. This is due to the fact that the boundary of the feasible set changes with $ \u $ and when the solution $ \xmin (\u) $ lies at the boundary for some $ \u $, the subdifferential of $ f $ with respect to $ \u $ at $ (\xmin(\u), \u) $ is a shifted non-trivial cone, hence, in particular not single-valued. We explain this phenomenon more concretely in \sref{prob:DirDiff} (see also \fref{fig:IntroEx}).

As a remedy, we invoke standard results from convex duality of the function $f$ to derive the above mentioned differentiability property of the value function for a large class of optimization problems including \eqref{eq:Lasso}. In fact, beyond differentiability, we explore the formula 
\begin{equation*}
    \subdiff{}{p} (\u) = \arg\max_{\y \in \RP} \innerprod{\u}{\y} - \fc(0, \y) \,,
\end{equation*}
which expresses the convex (Fr\'echet) subdifferential of the value function $p$ at $\u$ as the set of solutions to a certain optimization problem that depends on the convex conjugate $\fc$ of $f$. For a jointly convex function $f$ in $(\x,\u)$, the validity of the formula is asserted under the weak assumptions that $p(\u)$ is finite (i.e. the infimum of $f(\cdot,\u)$ in \eqref{eq:PrimalProb} is finite) and $ \u \in \ri (\dom{p}) $ lies in the relative interior of the domain of $p$. Therefore, in these situations, the problem of differentiating the value function $p$ is equivalent to solving a convex optimization problem, which allows us to explore the large literature on convex optimization algorithms. 

Since single-valuedness of the subdifferential of $p$ implies differentiability, for example, strict convexity of $\y\mapsto \fc(0, \y)$ implies differentiability of $p$ without the need for $f(\x,\u)$ to be differentiable. Therefore our approach allows us to compute the variation (gradient) of the value function $p$ in situations for which commonly used direct differentiation strategies, for example based on automatic differentiation, cannot be applied. Nevertheless, even if the parametric function $f$ is sufficiently smooth, the flexibility to apply various (optimal) convex optimization algorithms for computing this derivative information compares favorably with those direct differentiation strategies. 

For the large class of optimization problems that we consider, we summarize algorithms with their convergence guarantees based on the properties of the objective function.

\paragraph{Remark.} Differentiation of the value function $p$ in \eqref{eq:PrimalProb} is not to be confused with differentiating the optimal solution mapping $\xmin(\u)$ with respect to $\u$.  Besides its usage in computing automatic and implicit gradient estimator, the argmin derivative is used in optimization layers, that is, neural networks whose output is given by solving an optimization problem \cite{AK17, AAB+19}. It is also required in bilevel optimization \cite{DKPK15}, a most well known application of which is gradient-based hyperparameter optimization or parameter learning \cite{Dom12, KP13, DVFP14}.

Another problem which is similar to ours is the differentiation of a function $ g(\x, \u) $ with respect to the parameter $ \u $ evaluated at a solution $ \xmin(\u) $ of a system of parametric nonlinear equations $ h(\x, \u) = 0 $, where $ g : \RN \times \RP \to \RM $ and $ h : \RN \times \RP \to \RN $ satisfy some regularity conditions. We can also replace the function $ g $ with a functional and the non-linear system with a parametric ordinary differential equation. The two problems are related (but not equivalent) because when $ f $ in \eqref{eq:PrimalProb} is continuously differentiable and has a minimium $ \xmin(\u) $ in $ \x $ for a given $ \u $, then $ h = \grad{\x}{f} $ and $ g = f $ with $ M = 1 $ and our goal is to differentiate $ p(\u) = f(\xmin(\u), \u) $. To differentiate $ g(\xmin(\u), \u) $ with respect to $ \u $, we can make use of Piggyback differentiation \cite{GF03} or the Adjoint-state method \cite{Pon61, PLA18}. These techniques find their use in solving constrained optimization problems (where constraints are often given as ODE's or PDE's) with various applications in Geophysics \cite{Ple06}, Medicine \cite{KFTC12} and Neural Networks \cite{CRBD18}.

\section{Problem Setting} \label{sec:PS}

We consider parametric optimization problems of type \eqref{eq:PrimalProb} and seek for computing $\nabla p(\u)$, i.e., the variation of the value function $p$ with respect to $\u$.  One of our major goals is to characterize the properties of $f$ for which various numeric differentiation strategies with theoretical convergence guarantees can be used. We emphasize differentiation strategies based on iterative algorithms and provide convergence rates. 

First, in Section~\ref{prob:AAIG}, we recall the most widely used approaches for smooth parametric functions $f$, and demonstrate their limitations for several examples in \sref{prob:DirDiff}. Therefore, in \sref{sec:DualGrad}, as a remedy for such situations, we leverage a well known result from convex duality theory for numerical estimation of the variation of the value function, which allows us to classify problem classes with convergence rates.

\subsection{Analytical, Automatic and Implicit Gradient Estimator} \label{prob:AAIG}

Ablin et al.~\cite{APM20} analyze three different methods for iterative derivative approximation of smooth parametric functions $f$, provide convergences rates and enlighten a super-efficiency phenomenon for the automatic differentiation strategy. We recall their results.

Let $f$ be twice continuously differentiable on $ \RNXRP $ and $ \xmin(\u) $ be the unique minimizer for every $ \u \in \RP $ such that $ \hess{\x}{f} (\xmin(\u), \u)$ is positive definite. From the Implicit Function Theorem, we derive that $ \xmin : \RP \to \RN $ is continuously differentiable with derivative $ \Du\xmin(\u) = \varphi (\xmin (\u), \u) $, where we define the mapping $ \varphi : \RNXRP \to \RNXP $ as:
\begin{equation} \label{eq:Dmin}
    \varphi (\x, \u) = -[\hess{\x}{f} (\x, \u)]^{-1} \grad{\x\u}{f} (\x, \u)\,.
\end{equation}
The value function and its gradient are then given by:
\begin{equation*}
    p(\u) = f (\xmin(\u), \u)\ \mathrm{and}\
    \grad{}{p}(\u) = \grad{\u}{f} (\xmin(\u), \u) \,.
\end{equation*}
The expression for $ \grad{}{p} $ follows from chain rule and the optimality condition $ \grad{\x}{f} (\xmin(\u), \u) = 0 $. The minimizer $ \xmin (\u) $ is estimated by an iterative optimization method which yields a sequence $ \seq{\x}{k} $ with limit $ \xmin(\u) $. In a realistic setting such a process is terminated after $ K $ iterations to yield a so-called sub-optimal solution $ \xK(\u) $ for each $ \u $. To compute $ \grad{}{p} $, we either substitute $ \xK(\u) $ in place of $ \xmin(\u) $ in the expression for $ \grad{}{p} $ to obtain the \textbf{analytic gradient estimator}:
\begin{equation} \tag{AnG} \label{eq:AnG}
    \gK_{1} (\u) \coloneqq \grad{\u}{f} (\xK, \u)
\end{equation}
or in the expression for $ p $ and then differentiate it with respect to $ \u $, assuming that the sequence $ \seq{\x}{k} $ is differentiable, meaning that the mapping between successive iterates and the dependence on the parameter $\u$ is differentiable, giving us the \textbf{automatic gradient estimator}:
\begin{equation} \tag{AuG} \label{eq:AuG}
    \gK_{2} (\u) \coloneqq [\Du\xK]^{T} \grad{\x}{f} (\xK, \u)\ + \grad{\u}{f} (\xK, \u) \,.
\end{equation}
The term $ \Du\xK $ is an estimator of $ \Du\xmin $ and is obtained by applying automatic differentiation on $ \xK $, hence the name. Using the expression in \eqref{eq:Dmin} to estimate $ \Du\xmin $ yields the \textbf{implicit gradient estimator}, i.e.,
\begin{equation} \tag{IG} \label{eq:IG}
    \gK_{3} (\u) \coloneqq [\varphi (\xK, \u)]^{T} \grad{\x}{f} (\xK, \u)\ + \grad{\u}{f} (\xK, \u) \,.
\end{equation}
Ablin et al.~\cite{APM20} provide following error bounds for \eqref{eq:AnG}, \eqref{eq:AuG} and \eqref{eq:IG}.

\begin{theorem} \label{thm:AAIG}
Let $ D \coloneqq D_{x} \times D_{u} \subset \RNXRP $ be compact and $ f $ be $ m $-strongly convex with respect to $ \x $ and twice differentiable over $ D $ with second derivatives $ \grad{\x\u}{f} $ and $ \hess{\x}{f} $ respectively $ L_{xu} $ and $ L_{xx} $-Lipschitz continuous. Then the first derivatives $ \grad{\u}{f} $ and $ \grad{\x}{f} $ are respectively $ L_{u} $ and $ L_{x} $-Lipschitz continuous and for $ \xk $ produced by $ \xkp \coloneqq \xk - \tau \grad{\x}{f} (\xk, \u) $ with $ \tau \leq 1/L_{x} $ and $ \omega \coloneqq 1 - m\tau $, following statements hold:

\begin{enumerate}[label=\textnormal{(\alph*)}]
    \item The analytic estimator converges and we have:
    \begin{equation*}
        \norm{\gK_{1} - \grad{}{p} (\u)}_{2} \leq L_{x} \norm{\xz - \xmin (\u)}_{2} \omega^{K}
    \end{equation*}
    \item The automatic estimator converges and for $ C_{k} \coloneqq \tau (L_{x} k + \omega/2) (L_{xu} + L_{1}L_{xx}) $ with $ \norm{\Du\xk}_{2} \leq L_{1} $ we have:
    \begin{equation*}
        \norm{\gK_{2} - \grad{}{p} (\u)}_{2} \leq C_{K} \norm{\xz - \xmin (\u)}_{2} \omega^{2K - 1}
    \end{equation*}
    \item The implicit estimator converges and for $ C \coloneqq (L_{xu} + L_{1}L_{xx}) / 2 + L_{2}L_{x} $ with $ \norm{\varphi (\xK, \u)}_{2} \leq L_{2} $ we have:
    \begin{equation*}
        \norm{\gK_{3} - \grad{}{p} (\u)}_{2} \leq C \norm{\xz - \xmin (\u)}_{2} \omega^{2K}
    \end{equation*}
\end{enumerate}

\end{theorem}

\tref{thm:AAIG} shows faster convergence of automatic and implicit estimators as compared to analytical estimator. The automatic estimator is more stable than implicit estimator as depicted experimentally in \cite{APM20}. This makes the automatic method a strong contender for estimating $ \grad{}{p} $. It is also not computationally expansive thanks to reverse mode AD. The memory overhead is overcome by discarding the iterates $ \xk $ for $ k = 0, \dots, K-1 $ and using $ \xK $ only in all the calculations when going backward \cite{Chr94, MO20}. Ablin et al.~\cite{APM20} also study these methods under weaker conditions, for instance, when $ f(\cdot, \u) $ is $ \mu $-{\L}ojasiewicz \cite{AB09} which generalizes strong convexity. However their results depend on strong smoothness assumptions for $ f $.

\subsection{Problems with Direct Differentiation} \label{prob:DirDiff}

Obviously, the settings for which Ablin et al.~\cite{APM20} provide convergence rate guarantees is quite limited. We would like to emphasize the fact that differentiability of the parametric function $f$ is not required for that of the value function $ p $. In this section, we show with simple examples that the necessary conditions like differentiablity of the objective and the existence of the minimizer are the key disadvantages of the above methods.

\begin{example} \label{exmp:NonSmthToy}
Let $ f : \R \times \R \to \R $ be defined as:
\begin{equation*}
    f (x, u) = \exp(x) + \delta_{[u, +\infty)} (x) \,,
\end{equation*}
then \eqref{eq:AnG} and \eqref{eq:IG} fail to converge for all $ u \in \R $ while \eqref{eq:AuG} converges only when $ \xtk > u $ for all $ k $ (eventually) and $ \Dut \xtk $ converges to $ \Dut \xtmin(u) $.
\end{example}

\paragraph{Detail.} $ f $ is jointly convex in $ x $ and $ u $ and for all $ u \in \R, f(\cdot, u) $ is $ \exp(u) $-strongly convex and $ \xtmin : \R \to \R $ and $ p : \R \to \R $, given by $ \xtmin(u) = u $ and $ p(u) = \exp(u) $ respectively, are continuously differentiable on $ \R $. On the other hand, $ f $ is neither differentiable with respect to $ x $ nor $ u $ at $ (\xtmin(u), u) $ for any $ u \in \R $. To see why $ f $ is not differentiable with respect to $ u $, note that $ f $ is alternatively written as $ f(x, u) = \exp(x) + \delta_{(-\infty, x]} (u) $. The subdifferential of $ f $ with respect to $ u $ is $ \R_{+} $ when $ x = u $ and $ \{ 0 \} $ when $ x > u $. Thus when $ \xtk = u $ for some $ k \in \N $, none of the above methods is useful here. If $ \xtk > u $ for all $ k $, we get $ \gtk_{1} (u) = 0 $ and $ \gtk_{3} (u) = 0 $ since $ \partial f / \partial u (\xtk, u) = 0 $ and $ \partial^{2} f / \partial u^{2} (\xtk, u) = 0 $. The automatic estimator is given by $ \gtk_{2} (u) = \Dut\xtk (u) \exp (u) $ because $ \partial f / \partial x (\xtk, u) = \exp (u) $. It converges to $ p^{\prime} (u) = \exp (u) $ only if $ \Dut\xtk (u) $ converges to $ \Dut\xtmin (u) = 1 $.

The convergence of $ \Dut \xtk (u) $ to $ \Dut \xtmin (u) $ in \exref{exmp:NonSmthToy} is possible only under limited conditions which we do not establish here. The next example considers the non-smooth parametric objective of \eqref{eq:Lasso} in an analogue $ 1 $D setting.

\begin{example} \label{exmp:NonSmthLasso}
Let $ f : \R \times \R \to \R $ be defined as:
\begin{equation*}
    f (x, u) = \frac{1}{2} (a x - b)^{2} + \delta_{[-u, u]} (x) \,,
\end{equation*}
where $ a, b \in \R \backslash \{0\} $, then for all $ u \in (0, \abs{b/a}) $, \eqref{eq:AnG} and \eqref{eq:IG} fail to converge while \eqref{eq:AuG} converges only when $ \xtk \in (-u, u) $ for all $ k $ (eventually) and $ \Dut \xtk $ converges to $ \Dut \xtmin(u) $.
\end{example}

\paragraph{Detail.} $ f $ is jointly convex in $ x $ and $ u $ and for all $ u \in (0, \abs{b/a}), f(\cdot, u) $ is $ a^{2} $-strongly convex and we have $ \xtmin(u) = \sgn{b/a} u $ and $ p(u) = (a \xtmin (u) - b)^{2}/2 $. Since $ f(x, u) = (a x - b)^{2}/2 + \delta_{[\abs{x}, \abs{b/a})} (u) $ and $ \subdiff{u}{f} (x, u) = N_{[\abs{x}, \abs{b/a})} (u) $, $ f $ is not differentiable with respect to $ u $ at $ (\xtmin(u), u) $ for any $ u \in (0, \abs{b/a}) $. Given a sequence $ \xtk \in (-u, u) $ with limit $ \xtmin(u) $ we have $ \gtk_{1} = 0 $ and $ \gtk_{3} = 0 $ because $ \partial f / \partial u (\xtk, u) = 0 $ and $ \partial^{2} f / \partial u^{2} (\xtk, u) = 0 $. The automatic estimator $ \gtk_{2} (u) = \Dut\xtk (u) a (a \xtk - b) $ converges to $ p^{\prime} (u) = \Dut\xtmin (u) a (a \xtmin - b) $ when $ \Dut\xtk (u) $ converges to $ \Dut\xtmin (u) $.

\begin{example} \label{exmp:NoMin}
Let $ f : \R \times \U \to \R $ be defined as:
\begin{equation*}
    f(x, u) = \exp(x) + \frac{1}{2} u^{2} \,,
\end{equation*}
then \eqref{eq:AnG}, \eqref{eq:AuG} and \eqref{eq:IG} fail to converge for all $ u \in \R $.
\end{example}

\paragraph{Detail.} This case is obvious because $ \inf_{x} \exp(x) = 0 $ with minimum not attained, giving us $ p(u) = u^{2} / 2 $ and $ \argmin_{x} f(x, u) = \emptyset $.


While one may still use the methods of \cite{APM20} or cvxpylayers \cite{AK17, AAB+19} to efficiently estimate $ \grad{}{p} $ in situations like those presented above, it should be noted that differentiating the solution mapping is not a strictly more general approach than differentiating the value function. This is due to the failure of the applicability of the chain rule for evaluating $ \nabla_{\u} [f(\xmin(\u), \u)] $ for non-smooth functions in general \cite{BP20}. (The concept of a subgradient is not defined for a vector-valued non-smooth function and must be replaced by graphical derivatives and coderivatives; see \cite[Section 9.D]{RW98} and the chain rule in \cite[Theorem 10.49]{RW98}). This calls for a theoretically justified approach for estimating $ \grad{}{p} $ beyond those which are currently available \cite{APM20, AK17, AAB+19}.

\section{Dual Gradient Estimator} \label{sec:DualGrad}

The discussion in the previous section suggests that a different method is needed which is independent of directly differentiating the parametric objective function $ f $. Trading the differentiability assumption for a joint convexity assumption of $f$ in $(\x,\u)$, we invoke the powerful convex duality for computing derivative information of the value function $p$ in cases beyond differentiability of $f$. Moreover, the same statement provides an expression for the convex subdifferential of $p$. 

Denoting the convex conjugate of a function $p$ by 
\[
    p^*(\y) := \sup_{\u} \innerprod{\y}{\u} - p(\u) \,,
\]
and its biconjugate by $\pcc:=(p^*)^*$, the following result can be derived when strong duality, i.e., $\pcc=p$, holds. The dual of the problem defined in \eqref{eq:PrimalProb} is given by:
\begin{equation} \tag{$ \D $} \label{eq:DualProb}
    \pcc(\u) = \sup_{\y \in \RP} \innerprod{\u}{\y} - \fc(0, \y) \,.
\end{equation}
and for $ \u \in \ri{(\dom{p})} $ 
\begin{equation*}
    \subdiff{}{p} (\u) = \arg\max_{\y \in \RP} \innerprod{\u}{\y} - \fc(0, \y) \,.
\end{equation*}

When $ \subdiff{}{p} (\u) $ is single-valued, $ p $ is differentiable at $ \u $ and therefore solving \eqref{eq:DualProb} yields the gradient of the value function, which does not require differentiability of $ f $. These results rely on the following standard convex duality result, which we state from \cite[Theorem~4.1]{Dru20} \cite[Section~11.H]{RW98}.
\begin{theorem} \label{thm:Main}
For $ \X \subset \RN $ and $ \U \subset \RP $, let $ f : \X \times \U \to \Rext $ be a proper, lower semi-continuous and convex function. Then following are true for all $ \u \in \U $:
\begin{enumerate}[label=\textnormal{(\alph*)}]
    \item \textbf{Weak Duality:} $ \pcc(\u) \leq p(\u) $.
    \item \textbf{Subdifferential:} If $ p(\u) $ is finite, then:
    \begin{equation} \label{eq:thm:Subd}
        \partial p(\u) \subset \argmax_{\y \in \RP} \innerprod{\u}{\y} - \fc(0, \y) \,.
    \end{equation}
    If in addition, the inclusion $ \u \in \ri (\dom{p}) $ holds, then \eqref{eq:thm:Subd} holds with equality.
    \item \textbf{Strong Duality:} If the subdifferential $ \partial p(\u) $ is nonempty, then the equality $ \pcc(\u) = p(\u) $ holds and the supremum $ \pcc(\u) $ is attained.
\end{enumerate}
\end{theorem}

Therefore, \eqref{eq:DualProb} is key for computing the variation of $p$ with respect to $\u$. Our goal is reduced to the problem of solving \eqref{eq:DualProb}, for which the machinery of convex optimization can be invoked to state algorithms and convergence rates. Moreover, in contrast to the automatic differentiation strategy (backpropagation), there is no need to store the iterates, which dramatically reduces the memory requirements. 

Let $ \seq{\y}{K} $ be a sequence generated by an algorithm for solving \eqref{eq:DualProb}, we call the gradient computed by this method the \textbf{dual gradient estimator}:
\begin{equation} \tag{DG} \label{eq:DG}
    \gK_{4} (\u) = \yK \,.
\end{equation}
The dual estimator is computationally efficient since it requires solving an optimization problem and does not require computing any additional gradient and Hessian terms. The computational expenses depend only on the method used to solve the problem and the rate of convergence. Such an estimator also does not have a memory overhead like storing the iterates $ \seq{\y}{k} $.

\subsection{A Large Class of Parametric Optimization Problems}

As an application of our approach, we consider the following class of parametric optimization problem:
\begin{equation} \tag{$\RW_{f}$} \label{eq:Roc:obj}
    f(\x, \u) = \innerprod{\c}{\x} + h(\b - A\x + \u) + k(\x) \,,
\end{equation}
where $ h : \RP \to \Rext $ and $ k : \RN \to \Rext $ are proper, lower semi-continuous and convex, $ A : \RN \to \RP $ is a linear map and $ (\c, \b) \in \RNXRP $. The convex conjugate of $ f $ is given by:
\begin{equation*}
    \fc(\v, \y) = -\innerprod{\b}{\y} + \kc(\Aadj\y - \c + \v) + \hc(\y) \,,
\end{equation*}
which yields the conjugate of the value function as:
\begin{equation} \tag{$\RW_{\pc}$} \label{eq:Roc:vfc}
    \pc(\y) = -\innerprod{\b}{\y} + \kc(\Aadj\y - \c) + \hc(\y) \,.
\end{equation}
Therefore, in order to compute the variation of $p$ with respect to $\u$ by using \eqref{eq:thm:Subd}, we must solve a problem of the form:
\begin{equation} \label{eq:dual-grad-estimator-problem}
    \min_{\y\in\RP} \kc(\Aadj\y - \c + \v) + \hc(\y) - \innerprod{\b + \u}{\y} \,.
\end{equation}

In the following section, depending on the properties of $k$, $h$, and $A$, we provide algorithms and convergence rates for solving \eqref{eq:dual-grad-estimator-problem} and, hence, for approximating the variation of the value function $p$. As a generic algorithm for solving \eqref{eq:dual-grad-estimator-problem}, we mention the Primal--Dual Hybrid Gradient Algorithm by Chambolle and Pock~\cite{CP11} here.  A sufficient condition for uniqueness of solution of \eqref{eq:dual-grad-estimator-problem} is strong convexity of $ \hc $ which follows from the Lipschitz continuity of $ \grad{}{h} $. For a weaker condition, we state the following result:




\begin{proposition}
    Let $ h, k, A $ and $ \c $ in \eqref{eq:Roc:obj} be such that $ h $ is differentiable on $ \intr{(\dom{h})} $ and there exist $ (\x, \u) \in \dom{k} \times \intr{(\dom{h})} $ with $ \Aadj \grad{}{h} (\u) - \c \in \subdiff{}{k} (\x) $, then $ \subdiff{}{p} (\u) = \{ \grad{}{p} (\u) \} $ is single-valued for all $ \u \in \ri (A\dom{k} + \dom{h} - \b) $.
\end{proposition}


\begin{proof}
    The condition $ \Aadj \grad{}{h} (\u) - \c \in \subdiff{}{k} (\x) $ guarantees the existence of some $ \y \in \dom{\hc} $ with $ \Aadj \y - \c \in \dom{\kc} $. In such case, the expressions $ \hc(\y) $ and $ \kc(\Aadj \y - \c) $ are finite-valued and $ \dom{\pc} $ is non-empty. Since $ p $ is proper, lower semi-continuous and convex \cite[Theorem~3.101]{Hoh19}, for every $ \u \in \ri (\dom p) $ with $ \dom{p} = A\dom{k} + \dom{h} - \b $ \cite[Example~11.41]{RW98}, $ \subdiff{}{p} (\u) $ is non-empty \cite[Theorem~23.4]{Roc70}. The single-valuedness of $ \subdiff{}{p} (\u) $ then follows from the strict convexity of $ \hc $ (see \lref{lem:basic}\ref{basic:strict:diff}).
\end{proof}

To understand how this works we consider the example where $ h = \norm{\cdot}_{2}^{2}/2 $ and $ k = \lambda\norm{\cdot}_{2}^{2}/2 + \gamma \norm{\cdot}_{1} $ for $ \lambda > 0 $ and $ \gamma \geq 0 $. By choosing $ \u = 0 $ and $ \x $ as:
\begin{equation*}
    \x_{i} = \begin{cases}
      (-\c_{i} - \gamma)/\lambda &, \qquad\enspace\; \c_{i} < -\gamma \\[5pt]
      \qquad 0 &, -\gamma \leq \c_{i} \leq \gamma \\[5pt]
      (-\c_{i} + \gamma)/\lambda &, \enspace\; \gamma < \c_{i} \,,
  \end{cases}
\end{equation*}
we observe that $ \Aadj\grad{}{h}(\u) - \c = -\c \in \subdiff{}{k} (\x) $.

Parametric Optimization problems of the form \eqref{eq:Roc:obj} are ubiquitous in Machine Learning, Computer Vision, and Signal Processing. In Signal and Image Processing, the parameter $ \u $ represents the observed variable while the mapping $ A $ represents the operation performed on the optimal hidden variable $ \x $ (which is to be determined) to obtain the observed variable. In Machine Learning, $ \u $ is the target or the label vector, $ A $ represents the feature matrix obtained from the independent variable and $ \x $ denotes the weights of the mapping to be learned which fits the training set $ (A, \u) $. In this model, $ h $ measures the dissimilarity between $ A\x $ and $ \u $. The second term puts the penalty on $ \x $ and therefore indicates a prior information of the optimal $ \x $ which is necessary when $ P < N $. In many applications like supervised learning, image denoising and segmentation, $ N \leq P $, while in those like compressed sensing and deconvolution, $ N > P $.

Other than the above applications, \eqref{eq:Roc:obj} also generalizes the classical infimal convolution \cite{Roc70}. For example, such expressions occur in Image Processing applications in the context of regularization via Total Generalized Variation \cite{BKP10}. Moreover, the Moreau envelope \cite{RW98} of a non-smooth function is of the presented form and is employed for solving non-smooth optimization problems and is key for interpretation of many convex optimization algorithms such as proximal splitting methods \cite{LM79, CP11a}. Also, the penalty approaches for approximating the minimization of $ f(\x)+g(\x) $ via $ \min_{\x} ( f(\x) + \min_{\z} g(\z) + 1/2\norm{\x-\z}^{2} ) $ have the same form, which shows relations to alternating minimization approaches and has been employed in real world machine learning problems \cite{LWC19}.

\subsection{Rate of Convergence} \label{DG:RoC}

By invoking convex duality, as described in the previous section, the computation of the value function's variation is reduced to solving problems of type \eqref{eq:dual-grad-estimator-problem} for which a large literature of optimization algorithms is available for several special cases.

We consider the following situations:
\begin{enumerate}[label=\textnormal{(\alph*)}]
    \item Let \eqref{eq:dual-grad-estimator-problem} be a quadratic problem with matrix $ Q \in \R^{P \times P} $ and let $ L = \lambda_{\max} (Q) $ and $ m = \lambda_{\min} (Q) $, then \eqref{eq:DG} computed by using conjugate-gradient method converges like $ \BigOh (\omega^{K}) $ with $ \omega \coloneqq (\sqrt{L} - \sqrt{m})/(\sqrt{L} + \sqrt{m}) $ \cite[Section 1.6]{Ber99}.
    
    \textbf{Discussion.} Depending on whether $ P $ is smaller (resp. larger) than $ N $, this rate is better (resp. worse) than those provided in \tref{thm:AAIG} (see the first column of \fref{fig:Exp} for a comparison).
    
    \item Let $ \hc $ be possibly non-smooth with efficiently computable proximal mapping and $ \kc \circ \Aadj $ has an $ L $-Lipschitz continuous gradient then the following are true for solving \eqref{eq:dual-grad-estimator-problem} by using ISTA and accelerated proximal gradient descent (FISTA):
    \begin{itemize}
        \item \eqref{eq:DG} converges with ISTA \cite[Theorem~4.9]{CP16} and FISTA \cite[Theorem~3]{CD15} to $ \grad{}{p} (\u) $.
        \item If $ \hc $ and $ \kc \circ \Aadj $ are strongly convex with parameters $ \delta \geq 0 $ and $ \gamma \geq 0 $ and $ \mu = \delta + \gamma > 0 $, \eqref{eq:DG} converges to $ \grad{}{p} (\u) $ like $ \BigOh (\omega_{1}^{K}) $ with ISTA \cite[Theorem~4.9]{CP16} and like $ \BigOh (\omega_{2}^{K}) $ with FISTA \cite[Theorem~4.10]{CP16} where $ \omega_{1} = (1 - \tau \gamma)/(1 + \tau \delta) $ and $ \omega_{2} = 1 - \sqrt{\tau \mu / (1 + \tau \delta)} $.
    \end{itemize} \label{enum:(F)ISTA}
    \textbf{Discussion.} This general setting is beyond the theory that is provided by \tref{thm:AAIG}.
    
    \item Let $ \hc $ and $ k $ be possibly non-smooth and their respective proximal mappings can be computed efficiently, then the following are true for solving \eqref{eq:dual-grad-estimator-problem} by using Primal--Dual Hybrid Gradient Algorithm with $ L = \norm{\Aadj} $:
    \begin{itemize}
        \item \eqref{eq:DG} converges to $ \grad{}{p} (\u) $ \cite[Theorem~5.1]{CP16}.
        \item If either $ \hc $ or $ k $ are strongly convex, then \eqref{eq:DG} converges to $ \grad{}{p} (\u) $ like $ \BigOh (1/K^{2}) $ to $ \grad{}{p} (\u) $ \cite[Theorem~2]{CP11}.
        \item If $ \hc $ and $ k $ are both strongly convex with parameters $ \delta $ and $ \gamma $ respectively, then \eqref{eq:DG} converges to $ \grad{}{p} (\u) $ like $ \BigOh (\omega^{K/2}) $ to $ \grad{}{p} (\u) $ with $ \omega = (1 + \theta)/(2 + \mu) $ and $ \mu = 2\sqrt{\delta\gamma} / L $ \cite[Theorem~3]{CP11}.
    \end{itemize} \label{enum:PDHG}
    \textbf{Discussion.} Similarly, this setting is more general than that of \tref{thm:AAIG}.
\end{enumerate}

\paragraph{Remark.} For non-strongly convex settings in \eqref{eq:dual-grad-estimator-problem}, the sequence of values $ \pc(\yK) - \innerprod{\u}{\yK} $ converges like $ \BigOh (1/K) $ with ISTA \cite[Theorem~4.9]{CP16} and PDHG \cite[Theorem~1]{CP11} and like $ \BigOh (1/K^{2}) $ with FISTA \cite[Theorem~4.4]{BT09} i.e., we have a potentially accelerated rate of convergence of the objective values. However this rate does not directly translate to a convergence rate of the iterates and hence, to a rate for the convergence of the dual gradient estimator. Such a conclusion requires additional properties of the optimization problem, such as local strong convexity, error bounds, or growth conditions \cite{AB09, FGP15, DL18}.

In order to recognize the potential of the dual gradient approach based on properties of the primal functions in \eqref{eq:Roc:obj}, we trace the conditions for the various convergence rates listed above back to properties of the primal functions. These results are based on the following lemma. Their proofs can be found in most standard texts on Convex Analysis, e.g., \cite{HL12} or \cite{Roc70}.

\begin{lemma} \label{lem:basic}
Let $ g, h : \RP \to \Rext $ be proper, lower semi-continuous and convex functions, $ \d \in \RP $ and $ B : \RN \to \RP $ be a linear mapping. Let $ l : \RN \to \Rext $ be defined by $ l(\x) = g(B\x + \d) $. Then the following results hold:
\begin{enumerate}[label=\textnormal{(\alph*)}]
    \item If $ g $ is $ m_{g} $-strongly convex on $ \RP $ for $ m_{g} \geq 0 $, then $ l $ is $ \lambda_{\min} (\Badj B) m_{g} $-strongly convex on $ \RN $.
    \item If $ g $ and $ h $ are strongly convex on $ \RP $ with parameters $ m_{g} $ and $ m_{h} $ respectively, then $ g + h $ is $ m_{g} + m_{h} $-strongly convex on $ \RP $.
    \item If $ g $ has an $ L_{g} $-Lipschitz continuous gradient on $ \RP $ for $ L_{g} \in (0, +\infty) $, then $ l $ has a $ \lambda_{\max} (\Badj B) m_{g} $-Lipschitz continuous gradient on $ \RN $.
    \item If $ g $ and $ h $ have Lipschitz continuous gradients on $ \RP $ with parameters $ L_{g} $ and $ L_{h} $ respectively, then $ g + h $ has an $ L_{g} + L_{h} $-Lipschitz continuous gradient on $ \RP $.
    \item If $ g $ is differentiable on $ \Omega \coloneqq \intr{\dom{g}} $, then $ \gc $ is strictly convex on each convex subset $ C \subset \grad{}{g} (\Omega) $. \label{basic:strict:diff}
    \item $ g $ is $ m_{g} $-strongly convex on $ \RP $ if and only if $ \gc $ has a $ 1/m_{g} $-Lipschitz continuous gradient on $ \RP $. \label{basic:strong:lip}
    \item $ g $ has an $ L_{g} $-Lipschitz continuous gradient on $ \RP $ if and only if $ \gc $ is $ 1/L_{g} $-strongly convex on $ \RP $. \label{basic:lip:strong}
\end{enumerate}
\end{lemma}

Let us look at \eqref{eq:Roc:obj} when the regularity conditions given in \tref{thm:AAIG} are satisfied for $ f $. Let $ h $ and $ k $ be strongly convex with parameters $ m_{h} > 0 $ and $ m_{k} > 0 $ respectively and twice differentiable with Lipschitz continuous first and second derivatives. Let $ L_{h} $ and $ L_{k} $ be Lipschitz constants of $ \grad{}{h} $ and $ \grad{}{k} $ and let $ L_{A} = \lambda_{\max} (\Aadj A), m_{p} = \lambda_{\min} (\Aadj A) $ and $ m_{d} = \lambda_{\min} (A \Aadj) $ then $ f(\cdot, \u) $ is $ m_{h} m_{p} + m_{k} $-strongly convex and has an $ L_{h}L_{A} + L_{k} $-Lipschitz continuous gradient. Using these parameters, the optimal convergence rate for gradient descent is given by $ (L - m)/(L + m) $ \cite{Pol87}, which along with \tref{thm:AAIG} gives us the rates for the analytical, automatic and implicit estimators as $ \BigOh (\omega_{p}^{K}), \BigOh (K\omega_{p}^{2K}) $ and $ \BigOh (\omega_{p}^{2K}) $ respectively where we have:
\begin{equation*}
    \omega_{p} = \frac{(L_{h} L_{A} - m_{h} m_{p}) + (L_{k} - m_{k})}{(L_{h} L_{A} + m_{h} m_{p}) + (L_{k} + m_{k})} \,.
\end{equation*}
The strong convexity parameters of $ \kc \circ \Aadj $ and $ \hc $ are $ m_{d}/L_{k} $ and $ 1/L_{h} $. The Lipschitz constants of gradients of these functions are $ L_{A}/m_{k} $ and $ 1/m_{h} $. These parameters similarly give us the convergence rate for the dual estimator as $ \BigOh (\omega_{d}^{K}) $ for:
\begin{equation*}
    \omega_{d} = \frac{L_{h} m_{h} (L_{k}L_{A} - m_{k} m_{d}) + L_{k} m_{k} (L_{h} - m_{h})}{L_{h} m_{h} (L_{k}L_{A} + m_{k} m_{d}) + L_{k} m_{k} (L_{h} + m_{h})} \,.
\end{equation*}
Assuming $ A $ is full rank, the convergence rates depend on whether $ P $ is larger or smaller than $ N $.

The condition of strong convexity of $ h $ or $ k $ can be relaxed to non-strong convexity. The expression for convergence rate for primal problem will stay the same with $ m_{h} $ or $ m_{k} $ set to $ 0 $. For the dual problem, we make use of the results listed in \ref{enum:(F)ISTA} and \ref{enum:PDHG} to compute the rate. We note that the theoretical guarantees for the primal gradient estimators are difficult to establish beyond the strong convexity and twice continuous differentiability of $ f $ in \eqref{eq:Roc:obj}. On the other hand, the dual gradient estimator is quite powerful as it converges in a very broad setting and the convergence rates are theoretically justified.

\section{Experiments} \label{sec:Exp}

We compare the performance of the four different gradient estimators, i.e., \eqref{eq:AnG}, \eqref{eq:AuG}, \eqref{eq:IG} and \eqref{eq:DG} for estimation of $ \grad{}{p} (\u) $ in different settings. Therefore we fix $ N $ and run these methods for different values of $ P $ and for different choices of $ h $ and $ k $ in \eqref{eq:Roc:obj}. Changing $ P $ will affect $ L_{A}, \lambda_{\min} (\Aadj A) $ and $ \lambda_{\min} (A \Aadj) $ while changing $ h $ and $ k $ will modify $ L_{h}, L_{k}, m_{h} $ and $ m_{k} $. This also includes cases of non-differentiability of $ k $ and non-strong convexity of $ h $. For each problem and for each $ P $, we generate error plots of the sequences $ \gn_{i} $ for a given $ \u $. Since the convergence rates for methods \eqref{eq:AnG}, \eqref{eq:AuG} and \eqref{eq:IG} depend on that of the original sequence, we also show the plots for $ \xn $ for each of the examples.

We consider the following four examples to experimentally verify our observations:
\begin{equation} \label{eq:Exmps}
    \begin{aligned}
        f_{1}(\x, \u) &= \frac{1}{2} \norm{\u - A\x}_{2}^{2} + \frac{\lambda}{2} \norm{\x}_{2}^{2} \\
        f_{2}(\x, \u) &= h_{\delta}(\u - A\x) + \frac{\lambda}{2} \norm{\x}_{2}^{2} \\
        f_{3}(\x, \u) &= \frac{1}{2} \norm{\u - A\x}_{2}^{2} + \frac{\lambda}{2} \norm{\x}_{2}^{2} + \gamma \norm{\x}_{1} \\
        f_{4}(\x, \u) &= h_{\delta}(\u - A\x) + \frac{\lambda}{2} \norm{\x}_{2}^{2} + \gamma \norm{\x}_{1} \,,
    \end{aligned}
\end{equation}
where $ h_{\delta} : \RP \to \R $ in second and fourth equations in \eqref{eq:Exmps} is the  function defined by:
\begin{equation*}
    h_{\delta}(\u) \coloneqq \begin{cases} 
      \quad\; \frac{1}{2}\norm{\u}_{2}^{2} &, \quad \norm{\u}_{2} \leq \delta \\[10pt]
      \delta \Big( \norm{\u}_{2} - \frac{\delta}{2} \Big) &, \quad \norm{\u}_{2} > \delta \,.
  \end{cases}
\end{equation*}
The conjugate of the corresponding value functions is given by:
\begin{align}
    \pc_{1}(\y) &= \frac{1}{2\lambda} \norm{A^{T}\y}_{2}^{2} + \frac{1}{2} \norm{\y}_{2}^{2} \nonumber \\
    \pc_{2}(\y) &= \frac{1}{2\lambda} \norm{A^{T}\y}_{2}^{2} + \hc_{\delta}(\y) \nonumber \\
    \pc_{3}(\y) &= \kc(A^{T}\y) + \frac{1}{2} \norm{\y}_{2}^{2} \nonumber \\
    \pc_{4}(\y) &= \kc(A^{T}\y) + \hc_{\delta}(\y) \nonumber \,,
\end{align}
with conjugate of elastic-net term $ k \coloneqq \lambda\norm{\cdot}_{2}^{2} + \gamma\norm{\cdot}_{1} $ given by:
\begin{equation*}
    \kc(\v) = \sum_{i=1}^{N} \max (0, \abs{v_{i}} - \gamma)^{2} / (2 \lambda) \,.
\end{equation*}


\begin{figure*}

\begin{subfigure}{\figlen\textwidth}
    \centering
    \includegraphics[width=\linewidth]{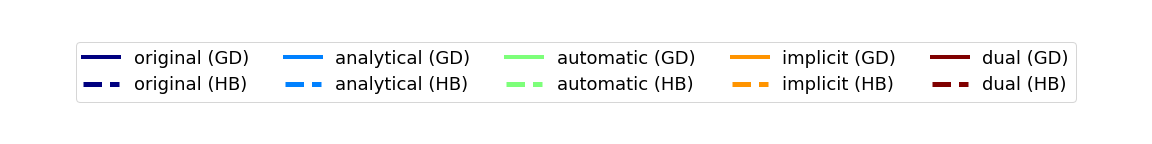}
\end{subfigure}

\begin{subfigure}{\figcelllen\textwidth}
    \centering
    \includegraphics[width=\linewidth]{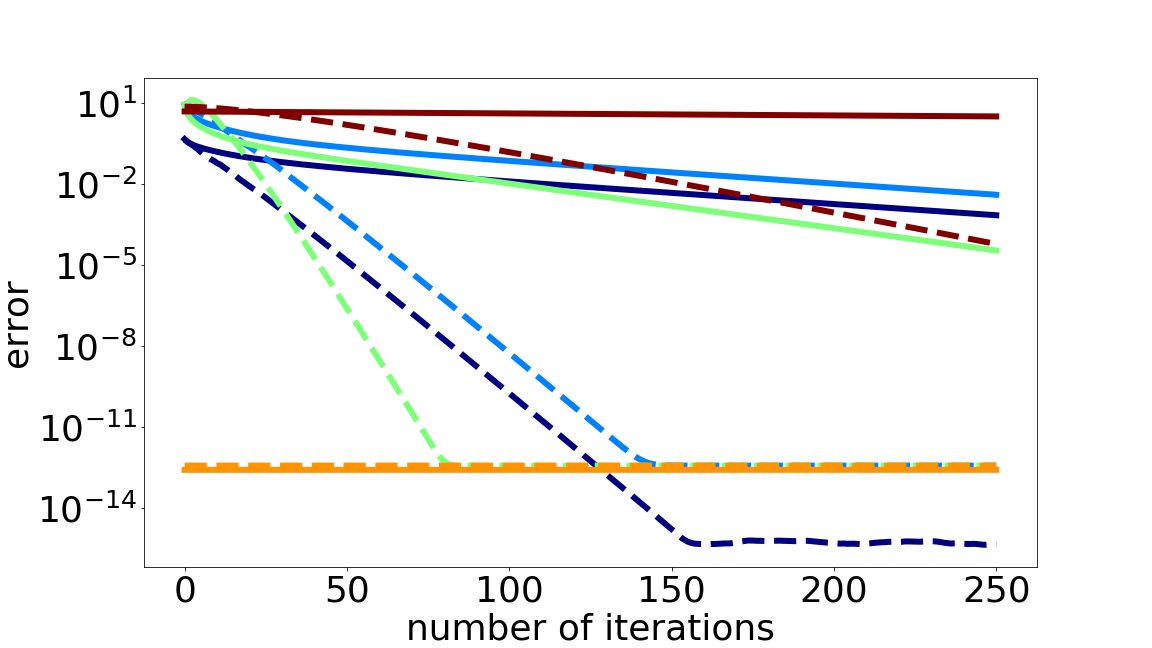}
\end{subfigure}
\begin{subfigure}{\figcelllen\textwidth}
    \centering
    \includegraphics[width=\linewidth]{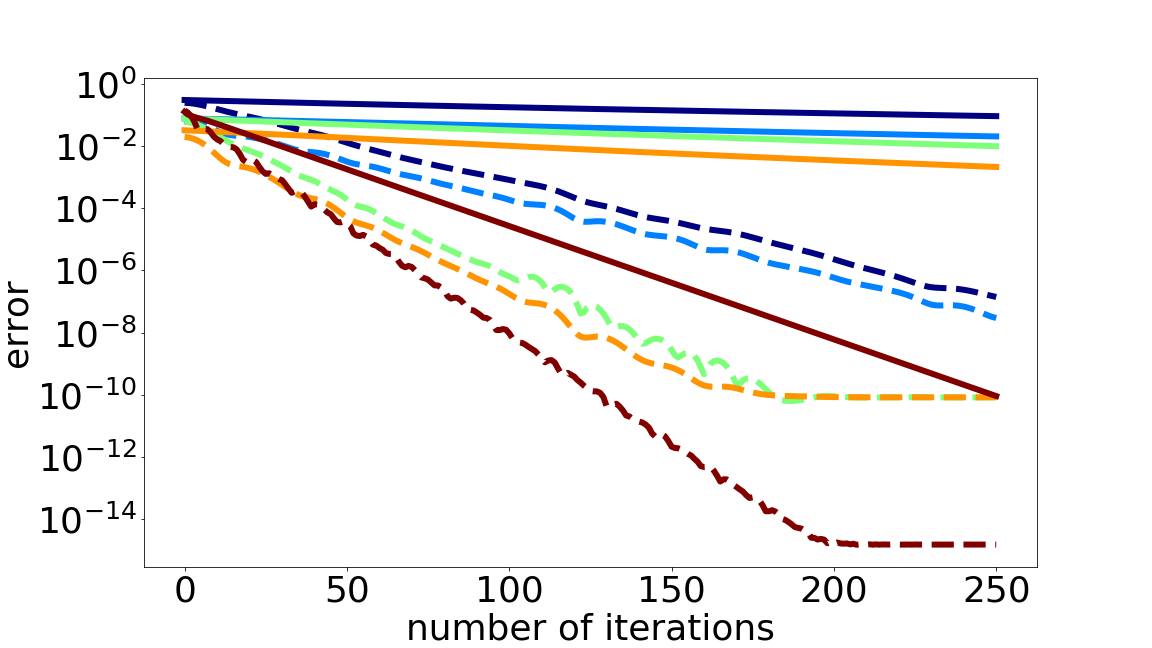}
\end{subfigure}
\begin{subfigure}{\figcelllen\textwidth}
    \centering
    \includegraphics[width=\linewidth]{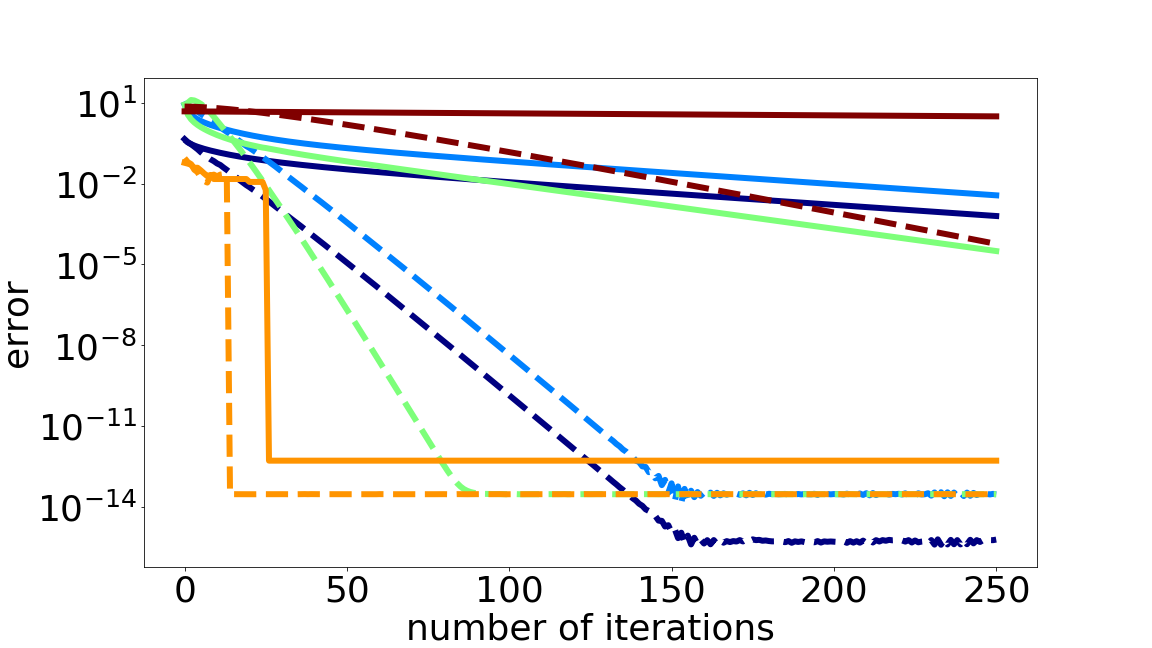}
\end{subfigure}
\begin{subfigure}{\figcelllen\textwidth}
    \centering
    \includegraphics[width=\linewidth]{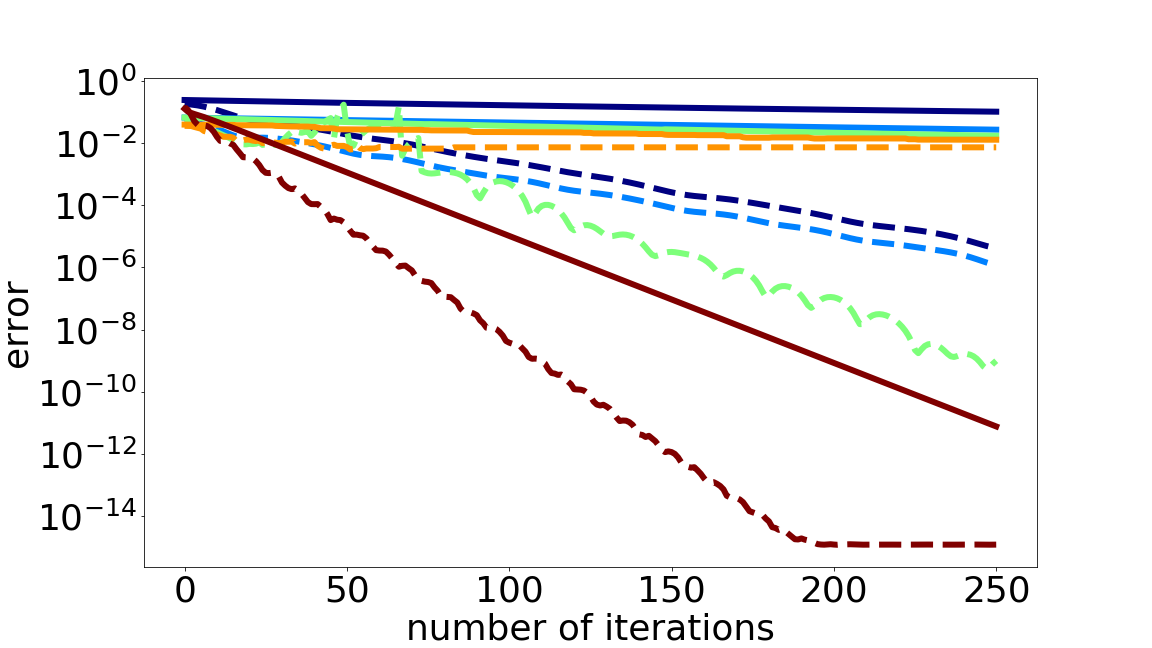}
\end{subfigure}

\begin{subfigure}{\figcelllen\textwidth}
    \centering
    \includegraphics[width=\linewidth]{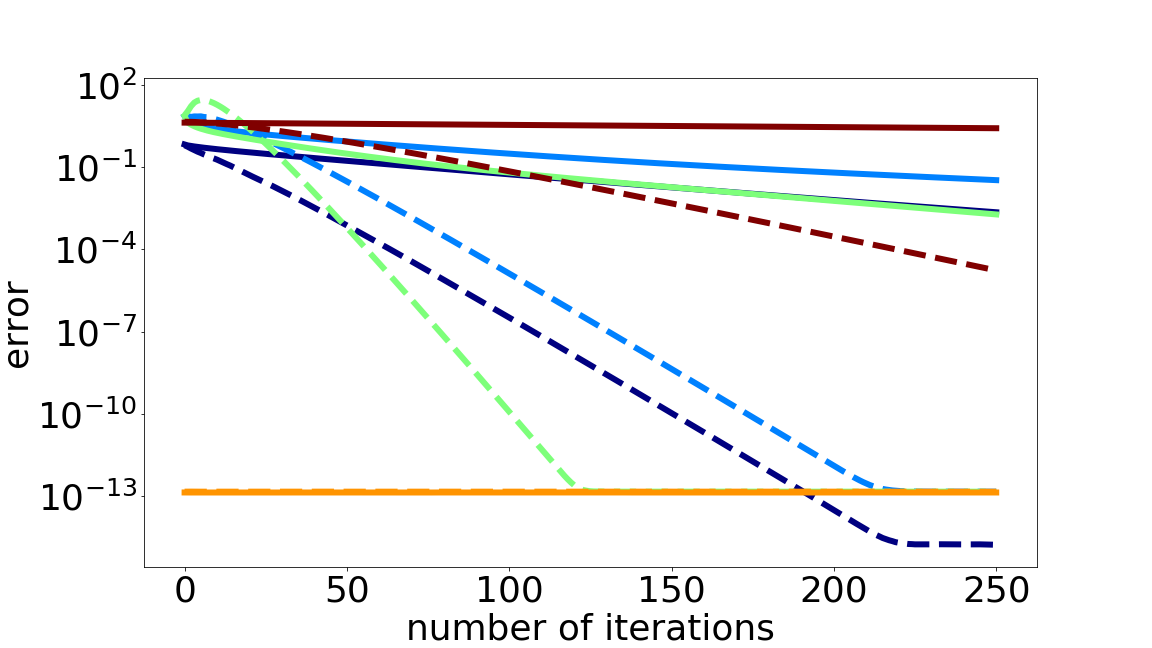}
\end{subfigure}
\begin{subfigure}{\figcelllen\textwidth}
    \centering
    \includegraphics[width=\linewidth]{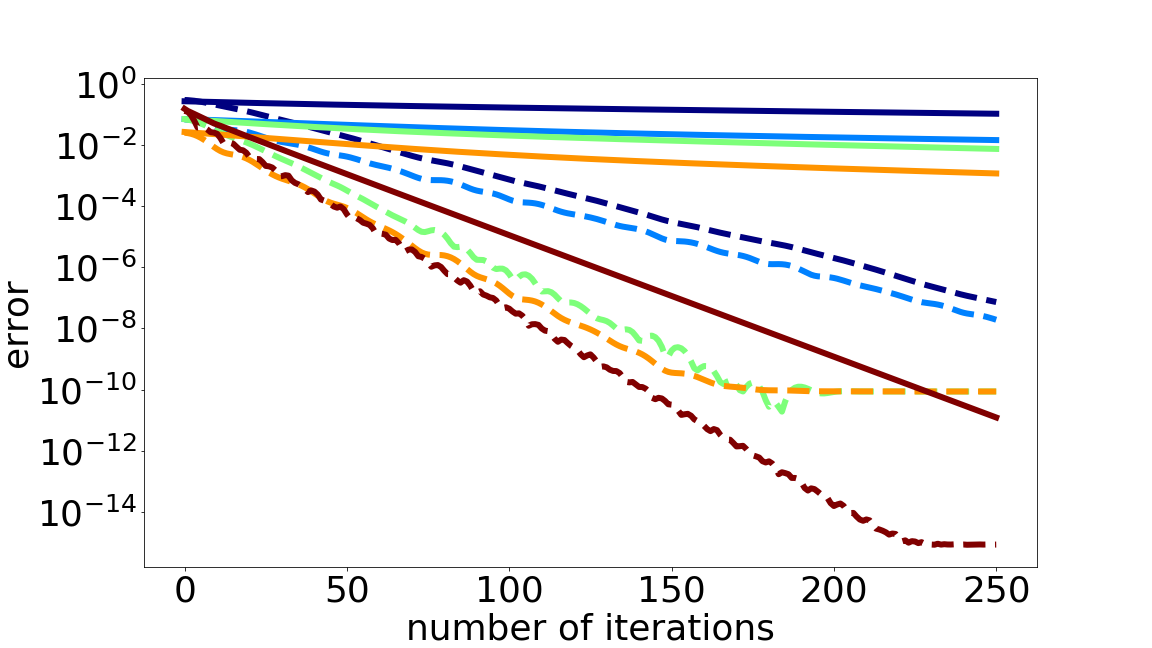}
\end{subfigure}
\begin{subfigure}{\figcelllen\textwidth}
    \centering
    \includegraphics[width=\linewidth]{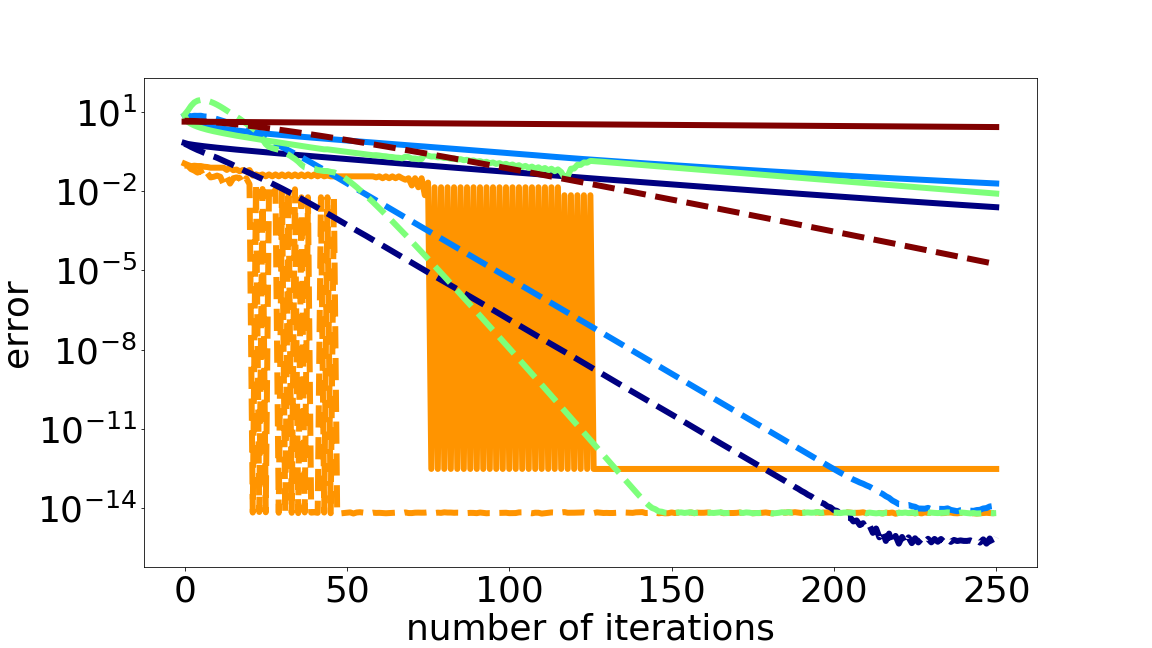}
\end{subfigure}
\begin{subfigure}{\figcelllen\textwidth}
    \centering
    \includegraphics[width=\linewidth]{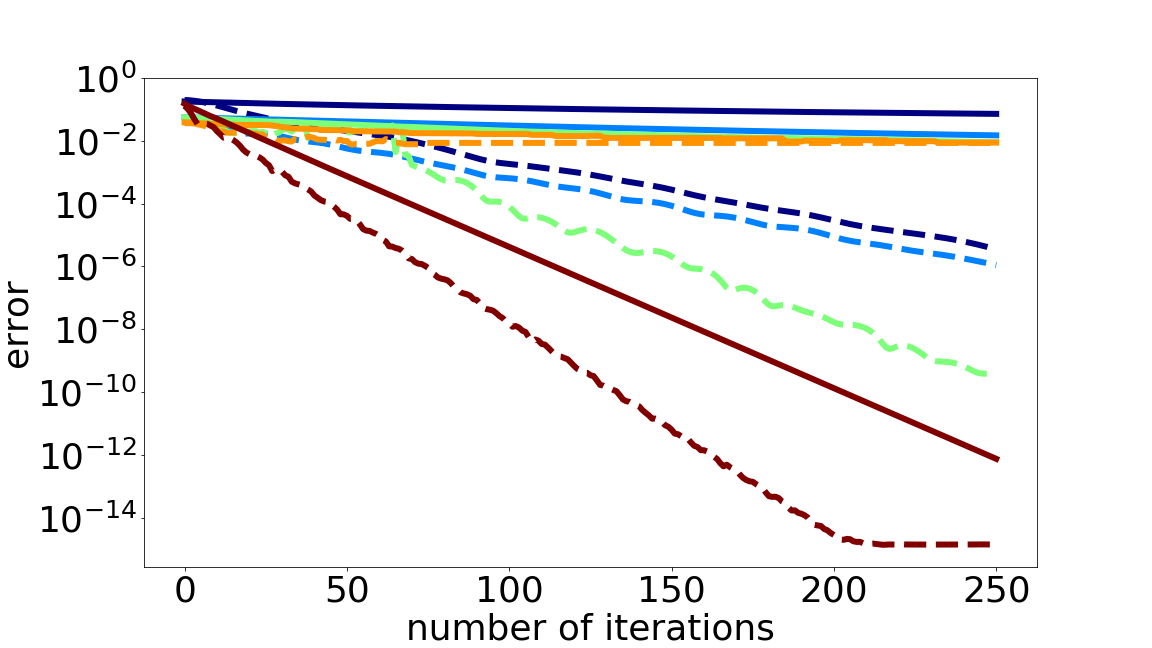}
\end{subfigure}

\begin{subfigure}{\figcelllen\textwidth}
    \centering
    \includegraphics[width=\linewidth]{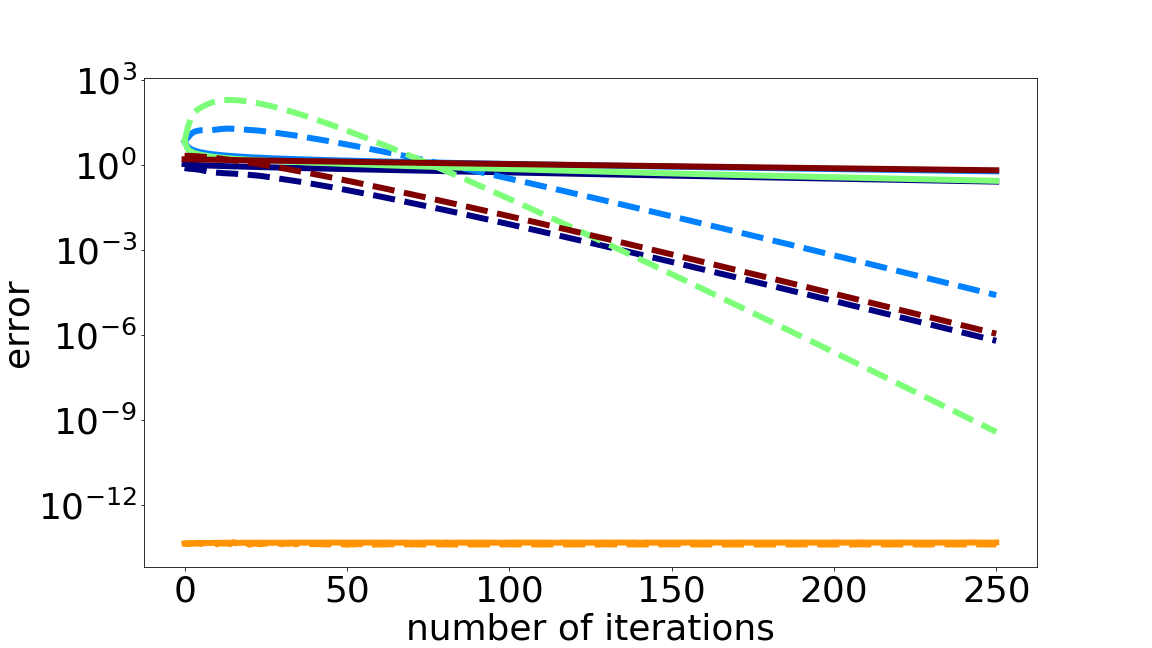}
\end{subfigure}
\begin{subfigure}{\figcelllen\textwidth}
    \centering
    \includegraphics[width=\linewidth]{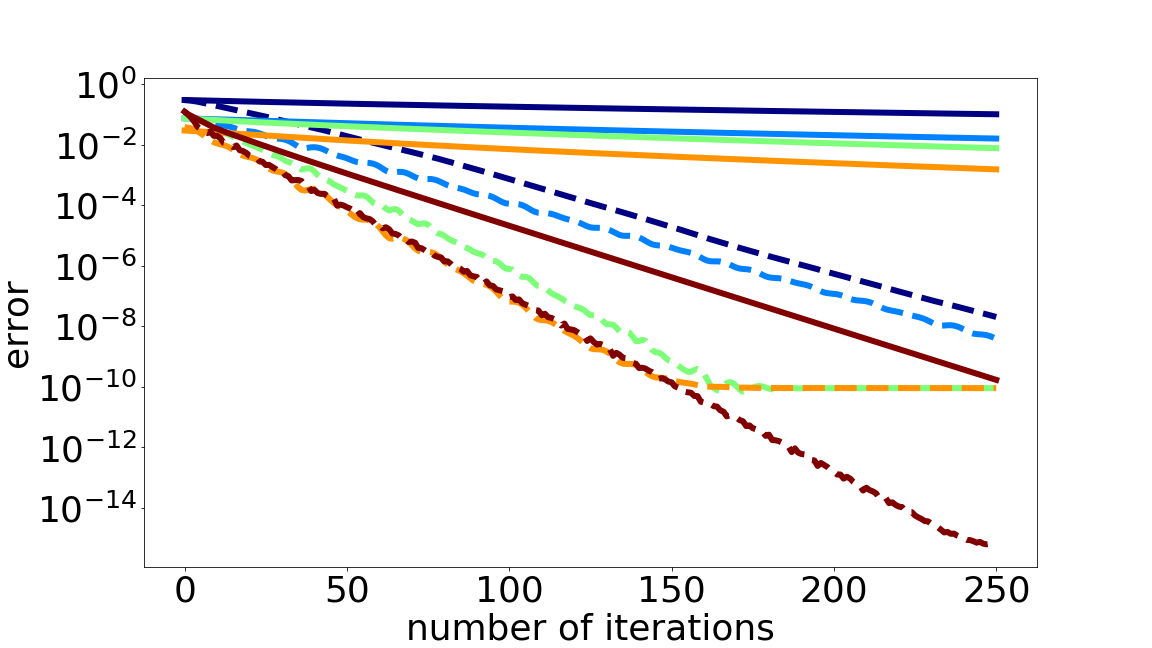}
\end{subfigure}
\begin{subfigure}{\figcelllen\textwidth}
    \centering
    \includegraphics[width=\linewidth]{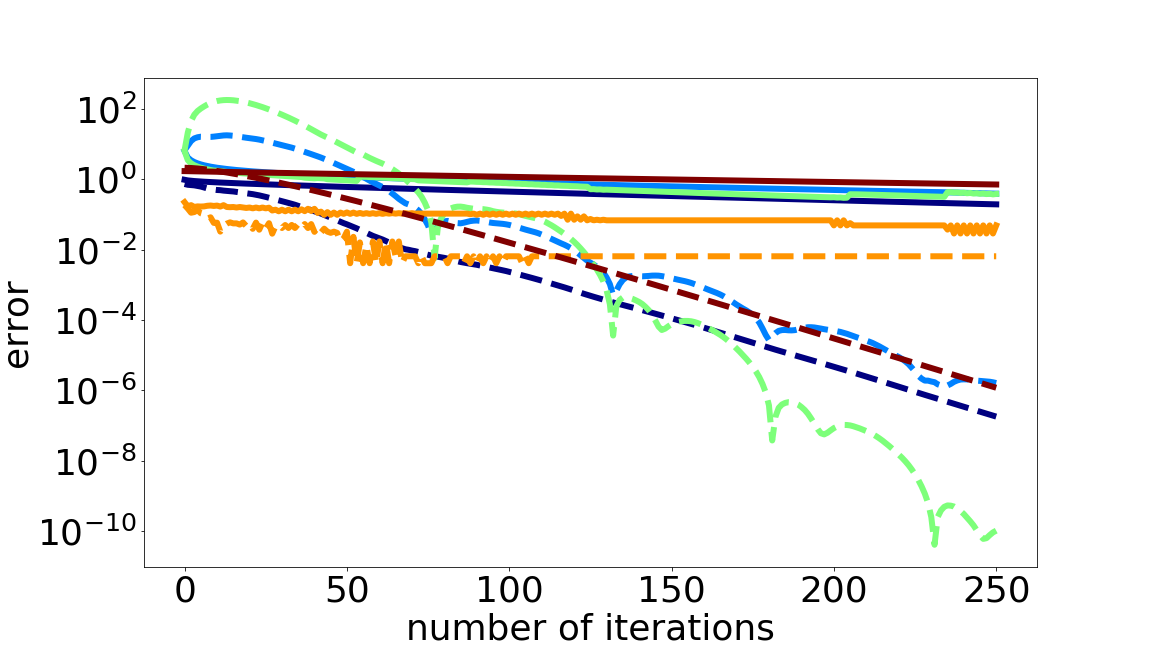}
\end{subfigure}
\begin{subfigure}{\figcelllen\textwidth}
    \centering
    \includegraphics[width=\linewidth]{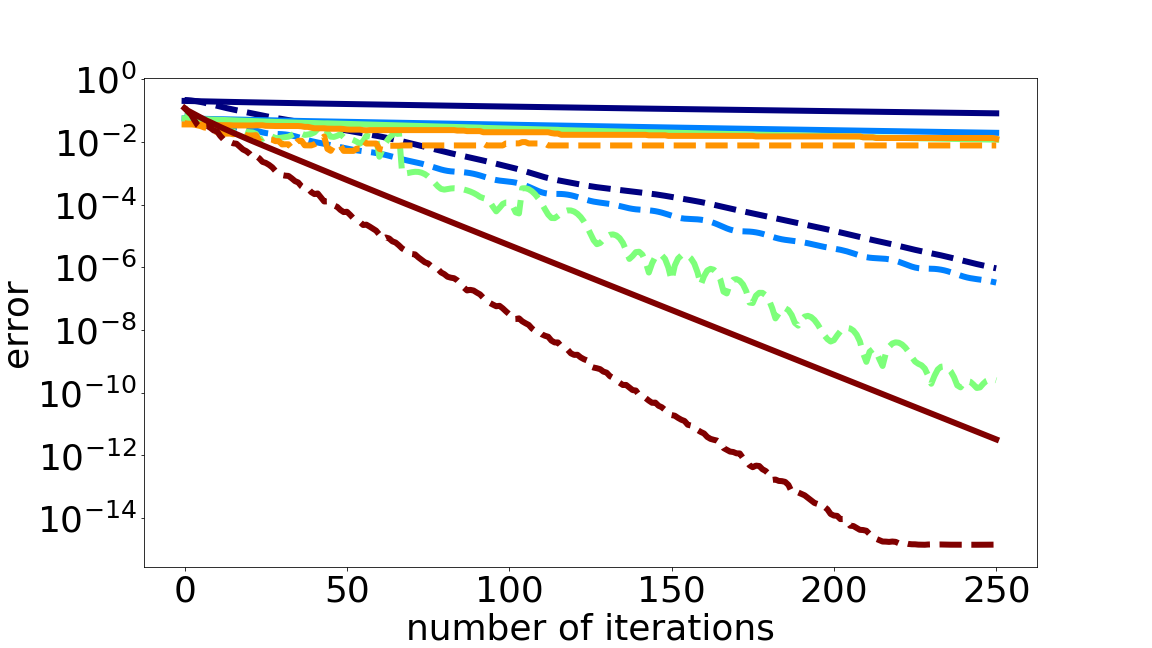}
\end{subfigure}

\begin{subfigure}{\figcelllen\textwidth}
    \centering
    \includegraphics[width=\linewidth]{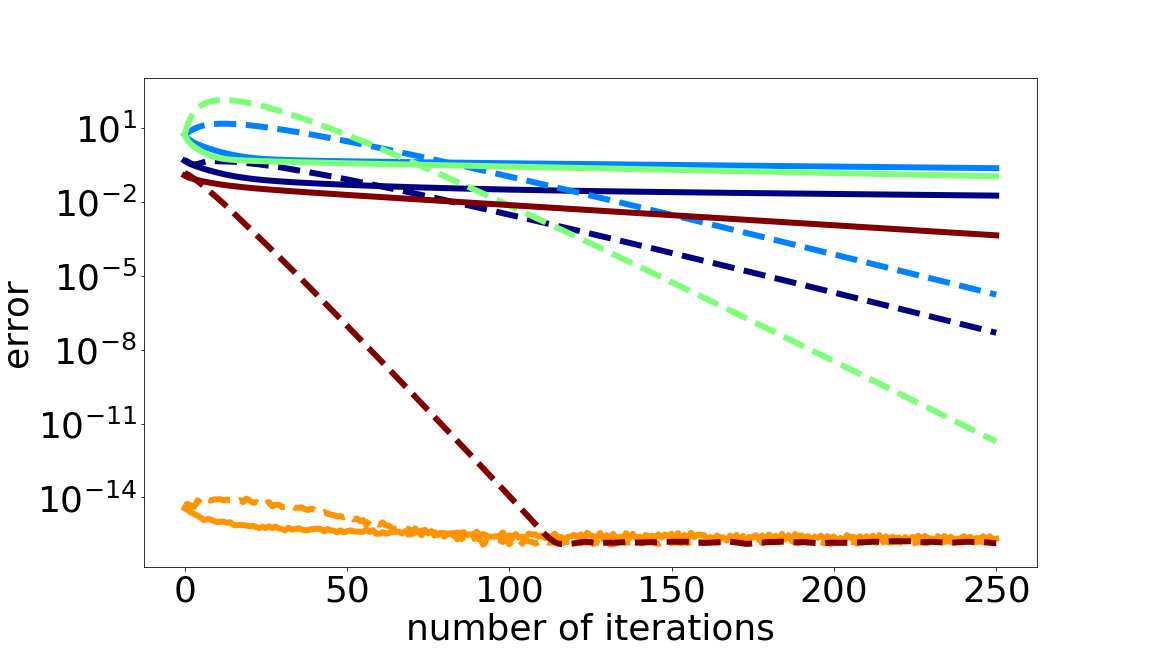}
\end{subfigure}
\begin{subfigure}{\figcelllen\textwidth}
    \centering
    \includegraphics[width=\linewidth]{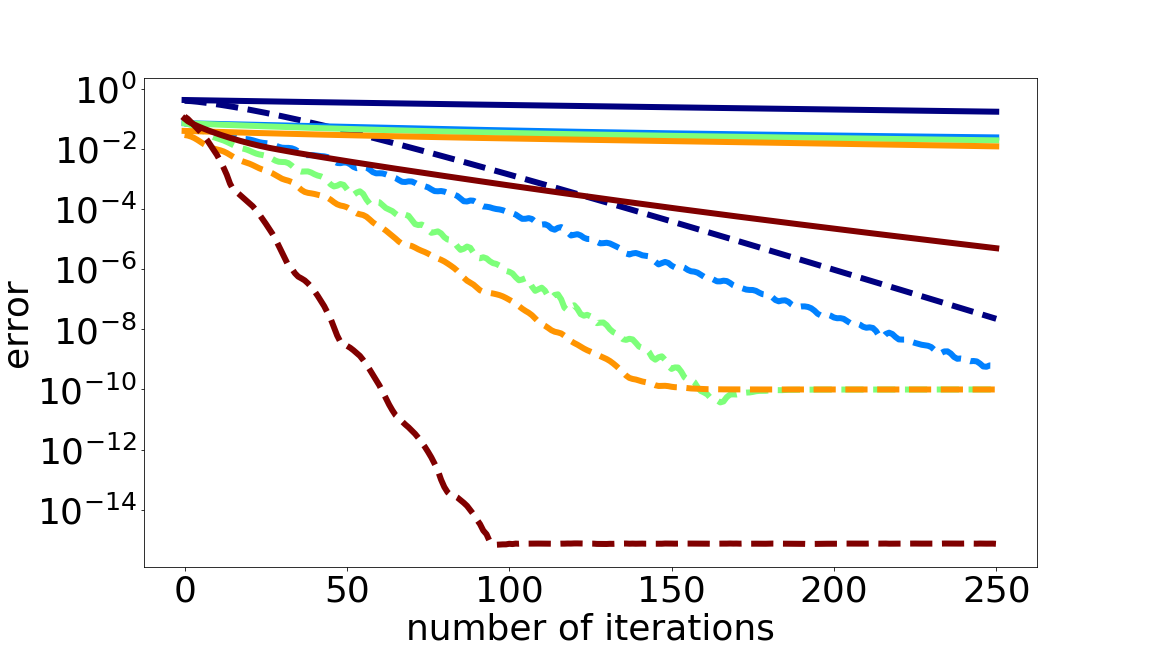}
\end{subfigure}
\begin{subfigure}{\figcelllen\textwidth}
    \centering
    \includegraphics[width=\linewidth]{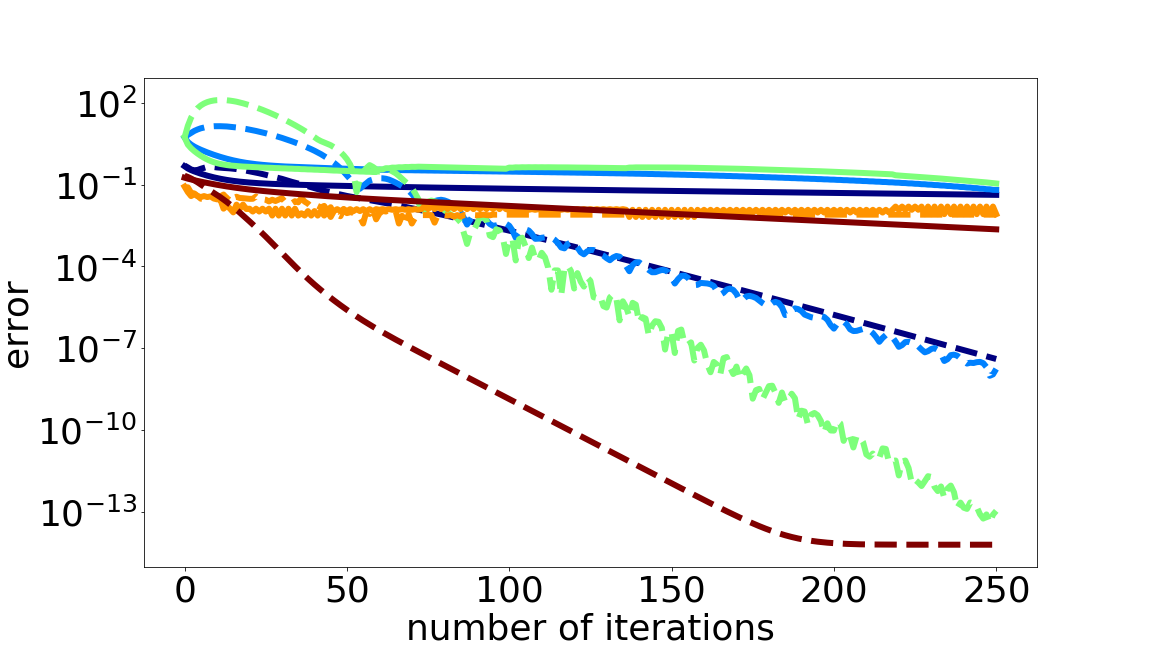}
\end{subfigure}
\begin{subfigure}{\figcelllen\textwidth}
    \centering
    \includegraphics[width=\linewidth]{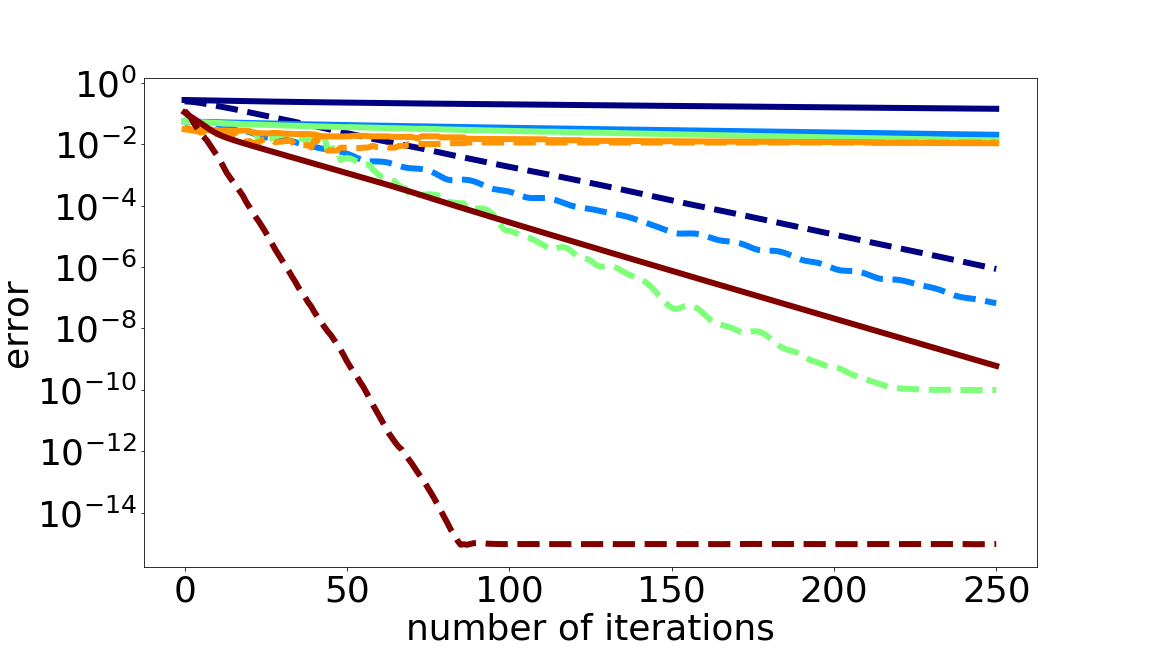}
\end{subfigure}

\begin{subfigure}{\figcelllen\textwidth}
    \centering
    \includegraphics[width=\linewidth]{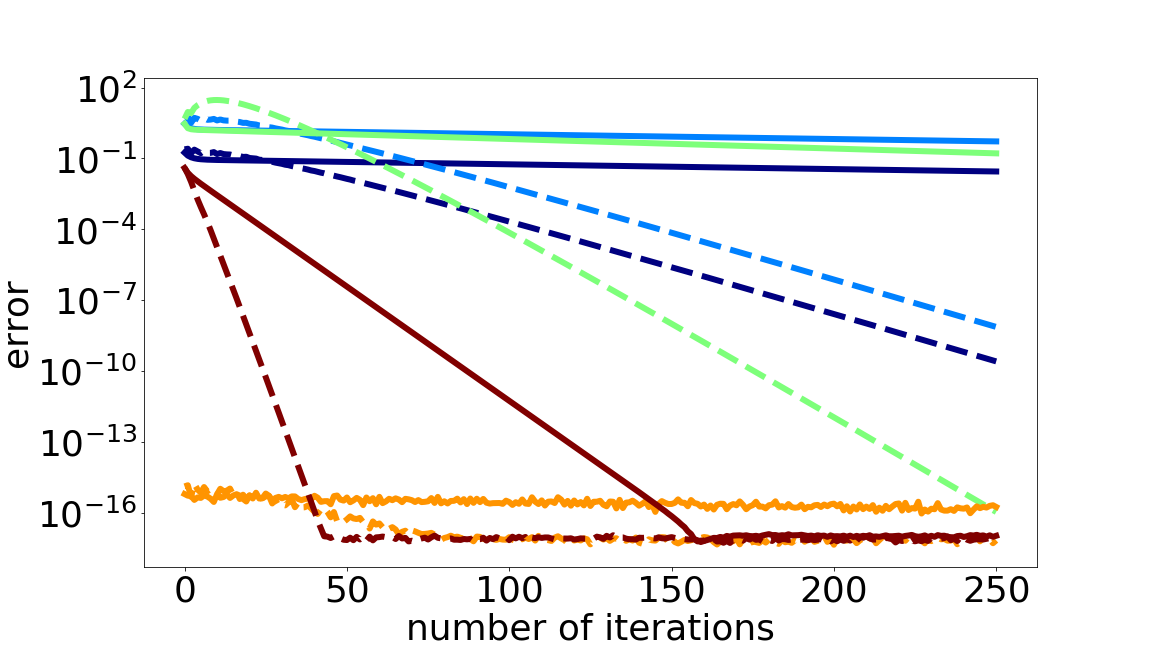}
\end{subfigure}
\begin{subfigure}{\figcelllen\textwidth}
    \centering
    \includegraphics[width=\linewidth]{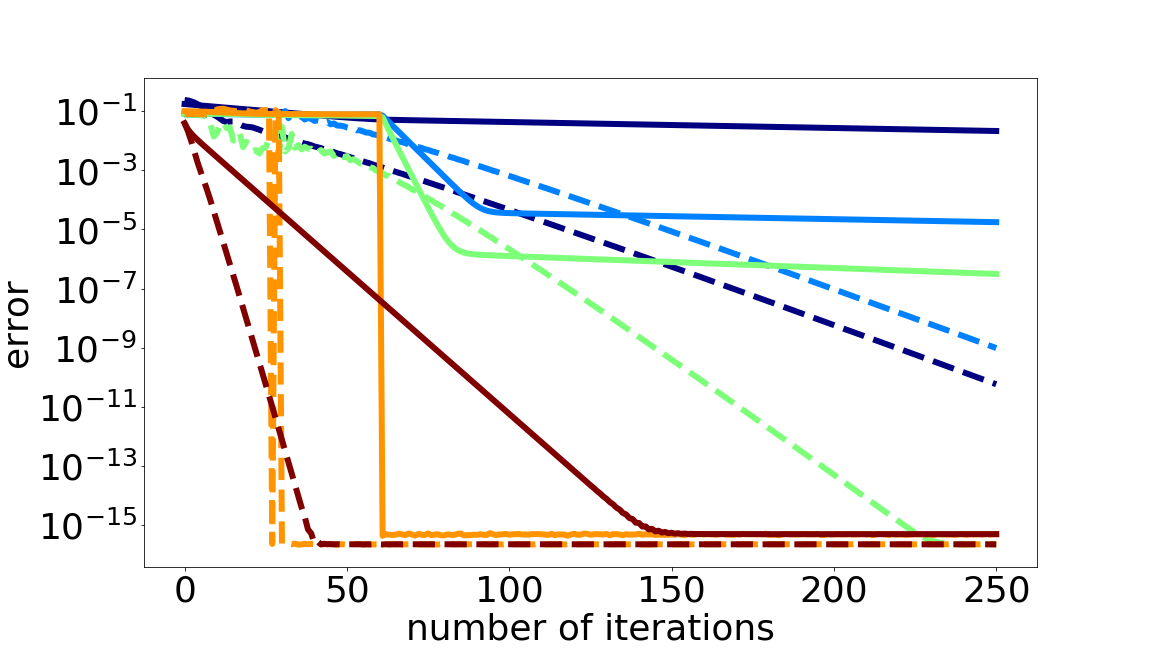}
\end{subfigure}
\begin{subfigure}{\figcelllen\textwidth}
    \centering
    \includegraphics[width=\linewidth]{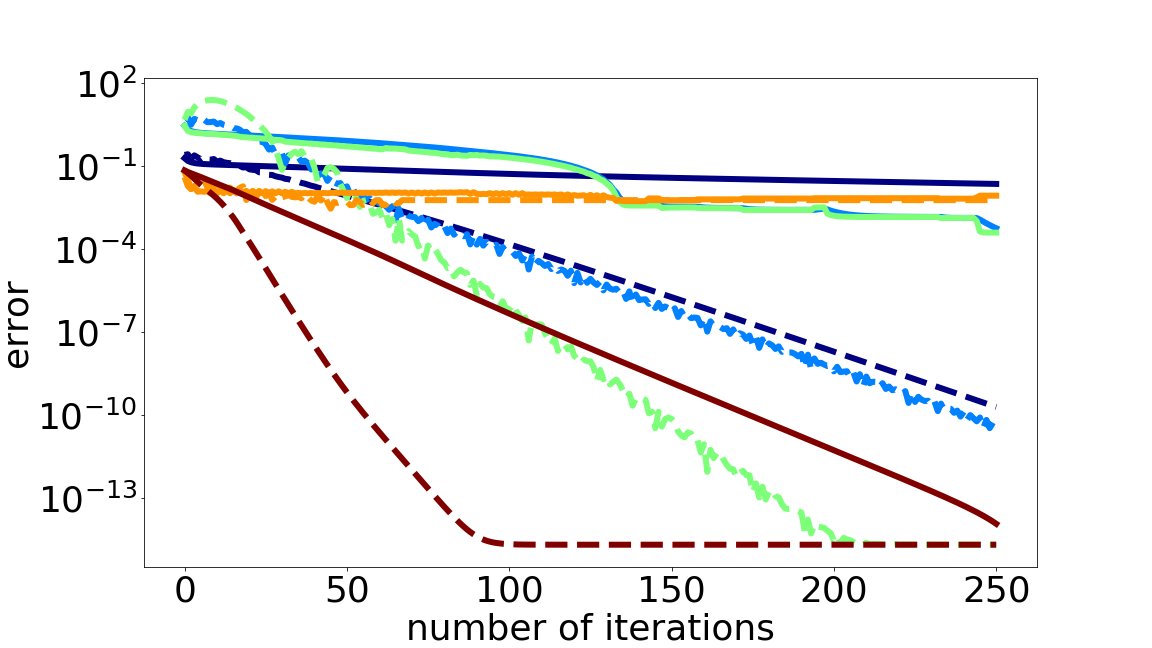}
\end{subfigure}
\begin{subfigure}{\figcelllen\textwidth}
    \centering
    \includegraphics[width=\linewidth]{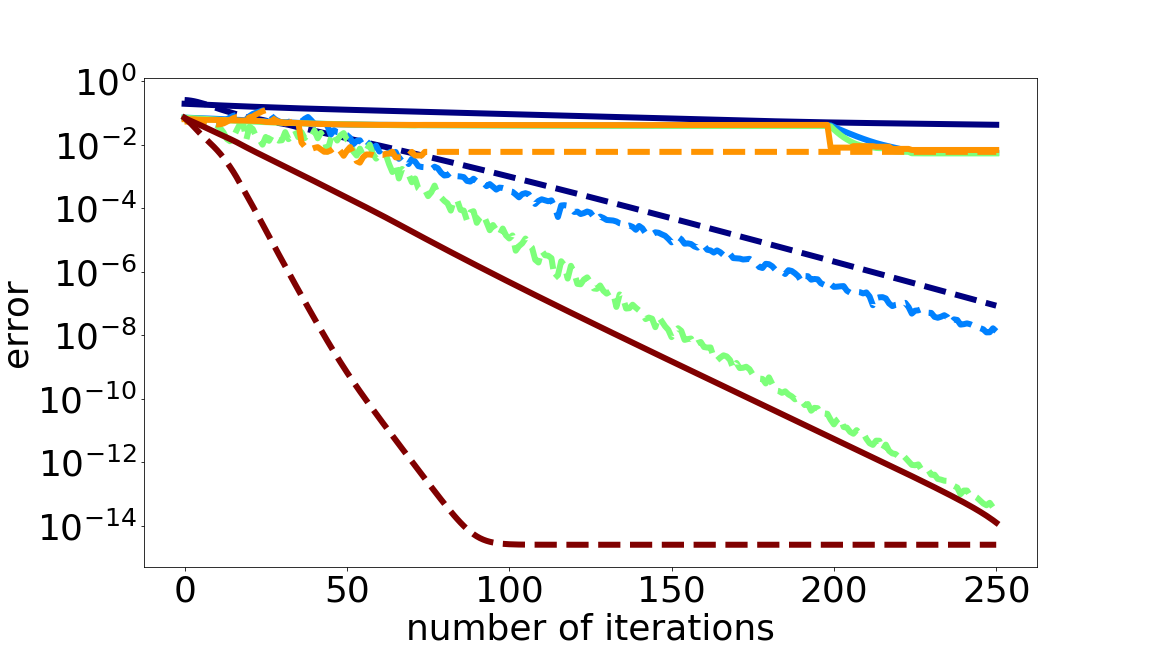}
\end{subfigure}

\caption{Error plots shown for the gradient sequences computed by the four methods, i.e., \eqref{eq:AnG}, \eqref{eq:AuG}, \eqref{eq:IG} and \eqref{eq:DG}, as well as the original (primal variable) sequence computed using (proximal) gradient descent (solid lines) and its inertial variant (dashed lines). The sequences are evaluated on four problems $ f_{1}, f_{2}, f_{3} $ and $ f_{4} $ ($ f_{j} $ changes from left to right in the given order) for five different values of $ P $, i.e., $ 90, 70, 50, 30 $ and $ 10 $ (P changes from top to bottom in the given order). Each cell shows the error plots for the ten sequences for a fixed $ P $ and $ f_{j} $; five for gradient descent or ISTA and five for the Heavy-ball method \cite{Pol64} or iPiasco \cite{OBP15}.}
\label{fig:Exp}
\end{figure*}

For evaluation we set $ N $ to $ 50 $ and choose $ P $ from $ \{ 10, 30, 50, 70, 90 \} $. This gives us five different plots for each problem and every value of $ P $ corresponds to a row in \fref{fig:Exp}. We set $ \lambda $ to $ 2 $, $ \gamma $ to $ 0.1 $ and $ \delta $ to $ 0.1 $. We keep $ \delta $ small because for sufficiently large values of $ \delta $, $ f_{2} $ will behave like $ f_{1} $ and $ f_{4} $ will behave like $ f_{3} $. Each element of $ A $ and $ \u $ is drawn from a normal distribution with mean $ 0 $ and standard deviation $ 1 $. Thus $ A $ is full-rank almost surely. We also scale each column of $ A $ differently to introduce ill-conditioning. For all our problems and methods, we use gradient descent when the problem is entirely smooth and proximal gradient descent when the problem has a non-smooth component with an optimal step size $ 2 / (L + m) $. To study the effect of inertia, we additionally employ the Heavy-ball method \cite{Pol64} and iPiasco \cite{OBP15} with optimal step size $ 4 / (\sqrt{L} + \sqrt{m})^{2} $ and momentum parameter $ (\sqrt{L} - \sqrt{m})^{2} / (\sqrt{L} + \sqrt{m})^{2} $ on these problems. The gradients and the Hessians are computed by using the autograd package \cite{MDA15}. For $ f_{1} $, an analytical expression exists for both $ \xmin (\u) $ and $ \grad{}{p} (\u) $. In order to compute a good estimate of such terms for the remaining problems, we run the primal and dual problems respectively for a large number of iterations. We verify the correctness of the obtained estimate of $ \grad{}{p}(u) $ by comparing it with the numerical gradient computed using central differences. Then we run each algorithm for $ 250 $ iterations for each $ P $ and $ f_{j} $ and generate the respective plots.

Each cell in \fref{fig:Exp} displays plots for $ \norm{\xn (\u) - \xmin (\u)}_{2} $ and $ \norm{\gn_{i} (\u) - \grad{}{p} (\u)}_{2} $ against the number of iterations $ K $. For $ f_{1} $ (first column), we note that all methods converge since they are backed by theoretical results. We see that for $ P > N $ (first column; first three rows), the dual method is slowest to converge and for $ P < N $ (first column; last two rows), it outperforms the analytical and automatic methods. Since the problem is quadratic therefore the implicit method yields $ \grad{}{p} $ in one step. For $ f_{2} $, the dual method shows faster convergence than all the methods for every choice of $ P $. The remaining two problems (third and forth columns) are not continuously differentiable and therefore \eqref{eq:AnG}, \eqref{eq:AuG} and \eqref{eq:IG} show an erratic behavior. The implicit method (red) performs very poorly in most cases. The analytical (orange) and automatic (green) gradient estimators manage to converge but do so in an irregular manner. Like $ f_{1} $, the dual method converges slowly when $ P \geq N $ for $ f_{3} $ and quickly when $ P < N $. Similarly, just like $ f_{2} $, the performance of the dual method is better than all other methods for every $ P $. The difference between the error plots generated by gradient descent or ISTA (solid lines) and the Heavy-ball method or iPiasco (dashed lines) is also visible. We observe that all the methods benefit from inertia. The fast convergence of automatic method is because of the fact that the acceleration in the convergence of $ \xK $ is also reflected in that of $ \Du\xK $ \cite{APM20, MO20}. In conclusion, we note that for the given non-smooth problems, especially $ f_{4} $, the dual gradient estimator is not only stable but also performs better than its primal counterparts.

\section{Conclusion}

The variation of the value function of a parametric optimization problem is desirable in a wide range of Machine Learning and Image Processing applications. The methods for computing this gradient usually rely on directly differentiating the objective and are thus limited to the settings when the objective satisfies strong smoothness conditions. We emphasize that the gradient of the value function can also be computed by using a well-known result from convex duality. This method provides an enormous flexibility for numerical approximation of the value functions derivative, allows to leverage convergence rate results from convex optimization algorithms, and does not rely on differentiability; It can compute a subgradient of the value function.

\section*{Acknowledgments}

Sheheryar Mehmood and Peter Ochs are supported by the German Research Foundation
(DFG Grant OC 150/4-1). 



{\small
\bibliographystyle{ieee}
\bibliography{main}
}

\end{document}